\documentclass[final]{amsart}

\usepackage{graphicx}
\usepackage{color}
\usepackage{amsmath,amsthm,amssymb,amsfonts}
\usepackage[foot]{amsaddr}
\usepackage{mathrsfs} 
\usepackage{url}
\usepackage{geometry} 
\usepackage{enumerate} 
\usepackage[color,notref,notcite]{showkeys} 
\usepackage[normalem]{ulem} 

\newcounter{cst}
\def \ctel#1{C_{\refstepcounter{cst}\label{#1}\thecst}} 
\def \cter#1{C_{\ref{#1}}} 

\newcommand*{\mathcolor}{}
\def\mathcolor#1#{\mathcoloraux{#1}}
\newcommand*{\mathcoloraux}[3]{%
  \protect\leavevmode
  \begingroup
    \color#1{#2}#3%
  \endgroup
} 

\definecolor{labelkey}{rgb}{0.6,0,1}
\definecolor{violet}{rgb}{0.580,0.,0.827}

\newtheorem{theorem}{Theorem}[section] 
\newtheorem{remark}[theorem]{Remark}
\newtheorem{lemma}[theorem]{Lemma}
\newtheorem{definition}[theorem]{Definition}
\newtheorem{proposition}[theorem]{Proposition}
\newtheorem{corollary}[theorem]{Corollary}
\numberwithin{equation}{section} 

\def\bhyp#1{\begin{equation}\label{#1}\begin{array}{c}}
\def\ehyp{\end{array}\end{equation}}

\def\Xint#1{\mathchoice
{\XXint\displaystyle\textstyle{#1}}%
{\XXint\textstyle\scriptstyle{#1}}%
{\XXint\scriptstyle\scriptscriptstyle{#1}}%
{\XXint\scriptscriptstyle\scriptscriptstyle{#1}}%
\!\int}
\def\XXint#1#2#3{{\setbox0=\hbox{$#1{#2#3}{\int}$ }
\vcenter{\hbox{$#2#3$ }}\kern-.6\wd0}}

\def\dashint{\Xint-}

\newcommand{\RR}{{\mathbb R}}
\newcommand{\NN}{{\mathbb N}}

\renewcommand{\O}{\Omega}
\newcommand{\eps}{\varepsilon}

\newcommand{\cF}{\mathcal{F}}

\newcommand{\ch}{\mathbf{1}}

\newcommand{\mmope}{{\mathcal T}}
\newcommand{\graph}{{\rm Gr}}

\newcommand{\un}{u_{n}}
\newcommand{\vn}{v_{n}}
\newcommand{\an}{a_{n}}
\newcommand{\fn}{f_{n}}

\newcommand{\zetan}{\zeta_{n}}
\newcommand{\betan}{\beta_{n}}
\newcommand{\nun}{\nu_{n}}
\newcommand{\Bn}{B_{n}}
\newcommand{\Tn}{T_{n}}
\newcommand{\Hn}{H_{n}}
\newcommand{\mun}{\mu_{n}}
\newcommand{\uini}{u^{\mathrm{ini}}} 
\newcommand{\uinin}{\uini_{n}}
\newcommand{\chin}{\chi_{n}}
\newcommand{\psin}{\psi_{n}}
\newcommand{\ol}[1]{\overline{#1}}

\newcommand{\chitilde}{\widetilde{\chi}}
\newcommand{\psitilde}{\widetilde{\psi}}
\newcommand{\mutilde}{\widetilde{\mu}}
\newcommand{\betatilde}{\widetilde{\beta}}
\newcommand{\zetatilde}{\widetilde{\zeta}}
\newcommand{\nutilde}{\widetilde{\nu}}
\newcommand{\atilde}{\widetilde{a}}
\newcommand{\weakto}{\rightharpoonup}

\newcommand{\weak}{\mbox{\rm-w}}
\newcommand{\trunc}{{\mathbb{T}}}

\def\term{{\mathcal I}}

\newcommand{\dive}{\operatorname{div}} 

\newcommand{\ud}{\, \mathrm{d}} 

\newcommand{\norm}[1]{\left\lVert#1\right\rVert} 

\newcommand{\Id}{\operatorname{Id}} 

\newcommand{\diam}[1]{\operatorname{diam}(#1)} 

\newcommand{\Leb}[3][ ]{L^{#2}(0, T; L^{#3}(\O)^{#1})} 

\newcommand{\LSob}[3]{L^{#1}(0,T; W^{#2,#3}(\O))} 

\newcommand{\LSobo}[3]{L^{#1}(0,T; W^{#2,#3}_{0}(\O))} 

\title[Convergence in $C(\lbrack0,T\rbrack;L^2(\O))$ for degenerate parabolic equations]{
Convergence in $C(\lbrack0,T\rbrack;L^2(\O))$ of weak solutions to perturbed doubly degenerate parabolic equations}

\author[J. Droniou]{J\'er\^ome Droniou$^1$}
	\address{$^{1,3}$School of Mathematical Sciences\\
	Monash University\\
	Clayton, Victoria 3800 Australia}
	\email{jerome.droniou@monash.edu}
\author[R. Eymard]{Robert Eymard$^2$}
	\address{$^2$Universit\'e Paris-Est\\
		Laboratoire d'Analyse et de Math\'ematiques Appliqu\'ees\\
		UMR 8050\\
		5 boulevard Descartes\\
		Champs-sur-Marne 77454 Marne-la-Vall\'ee Cedex 2\\
		France}
	\email{robert.eymard@u-pem.fr}	
\author[K.S. Talbot]{Kyle S. Talbot$^3$}
	\email{kyle.talbot@monash.edu}
	\thanks{$^3$Corresponding author}
	
\date{\today}

\begin{document}

\begin{abstract}
We study the behaviour of solutions to a class of nonlinear degenerate parabolic
problems when the data are perturbed. The class includes the Richards equation, 
Stefan problem and the parabolic $p$-Laplace equation. We show that, up to a subsequence, weak solutions
of the perturbed problem converge uniformly-in-time to weak solutions of the original problem
as the perturbed data approach the original data. We do not assume uniqueness or additional 
regularity of the solution. However, when uniqueness is known, our result demonstrates that 
the weak solution is uniformly temporally stable to perturbations of the data.
Beginning with a proof of temporally-uniform, spatially-weak convergence, we 
strengthen the latter by relating the unknown to an underlying convex structure
that emerges naturally from energy estimates on the solution.  The double degeneracy --- shown 
to be equivalent to a maximal monotone operator framework --- is handled with 
techniques inspired by a classical monotonicity argument and a
simple variant of the compensated compactness phenomenon. 
\end{abstract}			

\maketitle

\section{Introduction} \label{sec:intro}
Consider the class of doubly nonlinear parabolic problems
\begin{equation}\left\{
\begin{aligned} 
&\partial_t \beta(u)-\dive \left[ a(x,\nu(u),\nabla\zeta(u)) \right]= f &\quad &\mbox{in } \Omega\times(0,T),\\
&\beta(u)(x,0) = \beta(\uini)(x) &\quad &\mbox{in } \Omega,\\
&\zeta(u)=0 &\quad &\mbox{on } \partial\Omega\times(0,T)
\end{aligned}\right. \label{eq:model} \tag{P}
\end{equation}
on a bounded open subset $\O$ of $\RR^d$.
The functions $\beta$ and $\zeta$ are nondecreasing and 
the function $\nu$ satisfies $\nu'=\beta'\zeta'$. The operator $a$ is of Leray-Lions type
and $\uini\in L^2(\O)$.
In applications one may have only approximate knowledge of the data 
$(\beta,\zeta,\nu,a,f,\uini)$, and one is interested in the value of the solution
at a particular instant in time.
The main result of this article concerns the continuity of $\nu(u)$ with respect to 
perturbations of the data. For $n\in\NN$, consider perturbed problems
with corresponding solutions $\un$:
\begin{equation}\left\{
\begin{aligned} 
&\partial_t \betan(\un)-\dive \left[ \an(x,\nun(\un),\nabla\zetan(\un)) \right]= \fn &\quad &\mbox{in } \Omega\times(0,T),\\
&\betan(\un)(x,0) = \betan(\uinin)(x) &\quad &\mbox{in } \Omega,\\
&\zetan(\un)=0 &\quad &\mbox{on } \partial\Omega\times(0,T).
\end{aligned}\right. \label{eq:perturbed} \tag{$\mathrm{P}_{\!n}$}
\end{equation}
If the data $(\betan,\zetan,\nun,\an,\fn,\uinin)$ converge to 
$(\beta,\zeta,\nu,a,f,\uini)$ in suitable manners, we show that, up to a subsequence, 
$\nun(\un)$ converges to $\nu(u)$ in $C([0,T]; L^2(\O))$.

Instances of \eqref{eq:model} arise in various contexts. We focus our attention
upon three models in particular: the Richards equation, the Stefan problem, and the
parabolic $p$-Laplace equation. By taking $\zeta$ to be the identity, $\nu=\beta$ 
and $a(x,\nu(u),\nabla\zeta(u))=K(x,\beta(u))\nabla u$, we recover the first of these,
which describes the flow of water in an unsaturated porous medium \cite{mr10,rich31}.
The quantity of interest is the pressure-dependent saturation $\beta(u)$, with $K(x,\beta(u))$
the mobility. A model of the Stefan problem \cite{bdmp84} of heat diffusion in a medium undergoing
phase transition is realised by taking $\beta$ to be the identity, $\nu=\zeta$ and
$a(x,\nu(u),\nabla\zeta(u))=K(x,\zeta(u))\nabla\zeta(u)$. Here we are interested in
the enthalpy-dependent temperature $\zeta(u)$, with $K(x,\zeta(u))$ representing
the thermal conductivity. To recover the parabolic $p$-Laplace
equation, take each of $\beta$, $\zeta$ and $\nu$ to be the identity and 
$a(x,\nu(u),\nabla\zeta(u))= |\nabla u|^{p-2}\nabla u$. The parabolic $p$-Laplace equation
features in, for example, the theory of non-Newtonian filtration; see E. DiBenedetto's
monograph \cite{db93} and the references therein.

In each of these examples the quantity of practical interest is $\nu(u)$. More specifically,
it is the value of $\nu(u)$ at a particular instant in time, say $t=T$. Pragmatically
speaking, it is therefore critical that $\nu(u)(T)$ be stable to perturbations
of the data. Our main result shows this to be the case in each of the above examples,
where uniqueness of the solution is known (at least if
$K$ depends only upon $x$; see Appendix \ref{app:unicite}). For general problems \eqref{eq:model},
uniqueness appears to be open, so we can only assert that a subsequence
of $\nun(\un)$ converges to $\nu(u)$, where $u$ is a solution of the limit problem.

The existence and uniqueness of weak solutions to \eqref{eq:model} with $\zeta=\Id$ is studied in the seminal article 
of H.W. Alt and S. Luckhaus \cite{al83}. F. Otto \cite{otto} subsequently improved 
their uniqueness result by removing a linearity assumption on the diffusion operator
$a(\nu(u),\nabla u)$,
and by assuming independence with respect to $x$, strict monotony with respect to $\nabla u$ and
H\"older continuity with respect to $\nu(u)$.
However, to our knowledge there are no existence and uniform temporal--strong spatial 
stability results for parabolic equations with as many nonlinearities
and degeneracies as \eqref{eq:model}.

Stability results do exist for simplified models. 
Using techniques from nonlinear semigroup theory, P. B\'enilan and M.G. Crandall 
\cite{bc81} show that solutions to the Cauchy problem for $\partial_{t}u - \Delta\varphi(u)=0$ 
on the whole space are stable in $C([0,T]; L^{1}(\RR^{d}))$ with respect to pointwise 
perturbations of $\varphi$ and $L^{1}(\RR^d)$-perturbations of the initial datum.
D. Blanchard and A. Porretta \cite{bp05} demonstrate
the $\Leb{\infty}{1}$-stability of renormalized solutions to the initial-boundary
value problem for 
$\partial_{t}b(u) -\dive(a(x,u,\nabla u)) + \dive(\Phi(u))=f$, under $L^1$-perturbations
of the source and initial datum. The authors assume that $b$ is a maximal
monotone graph on $\RR$, $b^{-1}\in C(\RR)$ and $a$ is a Leray-Lions operator. 
We refer the reader to Section \ref{sec:maxmonop} for further comparisons of our
work to this reference.

Stability for other notions of solution to degenerate parabolic
problems has also been considered. In the framework of entropy solutions, 
B. Andreianov et al. \cite{abko09} demonstrate the stability in $L^{1}(\O\times(0,T))$ 
of solutions to \eqref{eq:model} with additional convection and reaction terms, but
with specific assumptions on the monotonicity of $\zeta$
and $a$. I.C. Kim and
N. Po\v z\'ar \cite{kimpoz13} show that viscosity solutions to the Richards equation
are stable. One can also consider stability of solutions to the parabolic $p$-Laplace equation
with respect to perturbations of $p$. To this end, we refer the reader to the work 
of J. Kinnunen and M. Parviainen \cite{kp10} and subsequently T. Lukkari and Parviainen \cite{lp15}.

The convergence of $\nun(\un)$ to $\nu(u)$ in $C([0,T];L^2(\O))$ 
cannot be deduced by mere interpolation
from the uniform-in-time $L^1(\O)$ stability results in the previous references,
since the best uniform-in-time bound that we can obtain for $\nu(u)$ is in $L^2(\O)$. 
From the viewpoint of uniform-in-time estimates, establishing a convergence result 
in this ``limit''  space $L^2(\O)$ therefore requires new ideas.
The first step is the uniform-$[0,T]$, weak-$L^{2}(\O)$ convergence
of $\betan(\un)$ to $\beta(u)$. A key ingredient of the proof of this fact, and indeed much of our paper,
is the function $B$ (and its perturbed analogue $\Bn$) defined below in \eqref{def:B}. The importance of
$B$ was previously observed in \cite{al83} when $\zeta=\Id$. It enables
energy estimates on the solution via an integration-by-parts formula for the
action of $\partial_{t}\beta(u)$ on $\zeta(u)$. These estimates are sufficient for us
to deduce the aforementioned convergence of $\betan(\un)$ thanks to Proposition 
\ref{prop:unifweakcomp}, a uniform-in-time, weak-in-space analogue of the 
Aubin--Simon compactness theorem. The spatial compactness is weak here since
\eqref{eq:model} does not provide any information on the gradient of $\beta(u)$.
The convexity of $B$ yields lower semi-continuity of certain integral functionals,
that when combined with the energy identity satisfied
by the limit solution, enables us to prove the uniform convergence of $\int_\O \Bn(\betan(\un)(x,\cdot))\ud x$ on $[0,T]$.
A uniform convexity property of $B$ connects the convergence of these integrals to that
of $\nun(\un)$ in $L^2(\Omega)$, thus enhancing the convergence of $\nun(\un)$ to 
prove the main result, Theorem \ref{th:main}.

We anticipate that these ideas for obtaining uniform-temporal, strong-$L^{2}$ spatial
dependence of solutions upon the data may generalise to systems of equations as in 
\cite{al83}, and to convection--diffusion--reaction equations of the form studied
in \cite{abko09}, but in the variational setting.

We obtain the existence of solutions to \eqref{eq:model} as a straightforward corollary
to Theorem \ref{th:main}. When 
$a(x,\nu(u),\nabla\zeta(u)) = \Lambda(x)\nabla\zeta(u)$, we give a short uniqueness
proof in Appendix \ref{app:unicite}. We do not, however, address uniqueness or regularity for
general $a$. With the nonlinearities in \eqref{eq:model} and the irregularities in the data 
seen in the applications described above, one cannot expect to obtain such properties
in these instances. Indeed, examples of non-uniqueness of weak solutions
exist, see \cite[Remark 3.4]{degh13} for stationary Leray--Lions equations (corresponding
to $\beta=0$ and $\zeta=\Id$).

Since $\beta$ and $\zeta$ may share common plateaux, one of the challenges in
studying compactness properties of solutions to \eqref{eq:model} is identifying
weak limits. Our method handles this difficulty principally using a monotonicity argument. 
However, the double degeneracy necessitates the use of a compensated compactness lemma 
(see Remark \ref{rem:compcomp}), which in our setting is actually
a direct consequence of the Aubin--Simon theorem.

These tools enable us to generalise some aspects of \cite{bp05}, 
at least when the regularity index $p$ is not too small; see the concluding remarks 
to Section \ref{sec:maxmonop} for additional discussion on this point.
The first two authors of the current article use similar techniques \cite{dey14} for the convergence
analysis of numerical approximations of \eqref{eq:model}. Discrete compensated compactness
was recently employed by B. Andreianov, C. Canc\`es and A. Moussa \cite{acm15} to identify 
the limits of numerical schemes in the framework of maximal monotone 
operators. 

The article is organised as follows. In Section \ref{sec:hypres} we list the hypotheses
on the model \eqref{eq:model} and state the main result, Theorem \ref{th:main}.  
In Section \ref{sec:maxmonop} we recast the problem in the framework of maximal monotone operators
and give the analogue of Theorem \ref{th:main} in this setting.
In Section \ref{ssec:B} we note some technical properties of the function $B$.
To focus attention on the convergence problem, some of these results are only stated. For proofs,
the reader should consult \cite{dey14}. Section \ref{ssec:est} establishes our estimates.
Section \ref{ssec:convmono} presents two lemmas that play
an important role in the proof of Theorem \ref{th:main}, and which may be of independent
interest. Our temporally-uniform, spatially-weak analogue of the Aubin--Simon compactness
theorem occupies Section \ref{ssec:unicomp}. Section \ref{sec:proof} is the proof of 
the convergence results, including the $C([0,T];L^2(\O))$ convergence.
Appendix \ref{app:convlemmas} lists several minor lemmas that we employ throughout the article. 
Aubin--Simon compactness appears again in Appendix \ref{app:cc}, where we use it to prove a
compensated compactness lemma adapted for our current setting. Appendix \ref{app:unicite}
is a self-contained uniqueness proof when the Leray-Lions operator $a$ is linear.

\section{Hypotheses and main result} \label{sec:hypres}

We assume that $T>0$, $\O$ is a bounded open subset of $\RR^d$ 
($d\in\NN$), and
\begin{subequations}\label{assumptions}
\bhyp{hyp:beta}
\mbox{$\beta:\RR\to\RR$ is nondecreasing, Lipschitz continuous with Lipschitz
constant}\\
\mbox{$L_\beta > 0$, and satisfies $\beta(0)=0$.}
\ehyp
\bhyp{hyp:zeta}
\mbox{$\zeta:\RR\to\RR$ is nondecreasing, Lipschitz continuous
with Lipschitz constant}\\
\mbox{$L_\zeta >0$, and satisfies $\zeta(0)=0$. Furthermore,
there are positive constants}\\ 
\mbox{$M_{1}, M_{2}$ such that for 
every $s\in\RR$, $|\zeta(s)|\geq M_{1}|s| - M_{2}$.}
\ehyp
\bhyp{hyp:nu}
\mbox{For all $s\in\RR$,}\quad
\displaystyle{\nu(s) = \int_{0}^{s}\zeta'(q)\beta'(q)\ud q.}
\ehyp
Fix $p\in(1,\infty)$ and denote by $p'=\frac{p}{p-1}$ its H\"older conjugate. 
We assume that $a:\O\times\RR\times\RR^d\to\RR^d$ is Carath\'eodory,
and that there are constants $\underline{a}, \mu>0$ and a function 
$\overline{a}\in L^{p'}(\O)$ such that for almost every $x\in\O$, every 
$s\in\RR$ and for all $\xi,\chi\in\RR^d$,
\bhyp{hyp:acoerc}
a(x,s,\xi)\cdot\xi\geq \underline{a}|\xi|^{p},
\ehyp
\bhyp{hyp:agrowth}
|a(x,s,\xi)|\leq \overline{a}(x) + \mu|\xi|^{p-1},
\ehyp
\bhyp{hyp:amono}
(a(x,s,\xi) - a(x,s,\chi))\cdot(\xi - \chi) \geq 0.
\ehyp
The source term and initial trace satisfy
\bhyp{hyp:uinif}
f\in\LSob{p'}{-1}{p'}, \quad \uini\in L^{2}(\O).
\ehyp
Due to the double degeneracy (from $\beta$ and $\zeta$), identifying weak
limits obtained by compactness results is challenging and requires
monotonicity and compensated compactness techniques. 
To prove that weak limits of sequences of solutions to \eqref{eq:model} are also
solutions to \eqref{eq:model}, we consider three separate cases for $p$:
\end{subequations}
\begin{equation}\left\{
	\begin{aligned}
	&\mbox{\rm (I)}&&\mbox{$p\geq 2$,}\\
	&\mbox{or}\\
	&\mbox{\rm (II)}&&\mbox{$\frac{2d}{d+2}<p<2$ and there are positive constants $M_{3}$,
	$M_{4}$ such that for all $s\in\RR$,}\\ 
	&&&|\beta(s)|\geq M_{3}|s| - M_{4},\\ 	
	&\mbox{or}\\
	&\mbox{\rm (III)}&&\mbox{$1< p\leq \frac{2d}{d+2}$, there are positive constants 
	$M_{3}$, $M_{4}$ such that for all $s\in\RR$,}\\ 
	&&&\mbox{$|\beta(s)|\geq M_{3}|s| - M_{4}$, and $\beta$ is (strictly) increasing.}
	\end{aligned}\right.\label{hyp:pbeta}
\end{equation}

\begin{remark} 
\begin{enumerate}[(i)]
\item The assumption $\beta(0)=\zeta(0)=0$ is not restrictive, since replacing
$\beta$ and $\zeta$ with $\beta-\beta(0)$ and $\zeta-\zeta(0)$ (respectively) does not change
the problem.
\item Hypotheses \eqref{hyp:acoerc} and \eqref{hyp:agrowth} can be relaxed to
\begin{equation*}
a(x,s,\xi)\cdot\xi\geq \underline{a}|\xi|^{p} -\Theta(x)\qquad \mbox{with $\Theta\in L^1(\O)$},
\end{equation*}
and
\begin{equation*}
|a(x,s,\xi)|\leq \overline{a}(x) + \mu |s|^q+ \mu|\xi|^{p-1}\qquad\mbox{ with $q<\max(2/p',p-1)$}.
\end{equation*}
\item The condition $p>\frac{2d}{d+2}$ in \eqref{hyp:pbeta} is equivalent to $p^*>2$,
where $p^*$ is the Sobolev exponent of $p$; i.e., $p^*=\frac{dp}{d-p}$ if $p<d$ and
$p^*=+\infty$ if $p\ge d$.
\item Since the basic energy estimates on \eqref{eq:model} provide strong
compactness for $\nu(u)$ (see \eqref{eq:nunstrong}), we can just as easily
handle source terms of the form $f(x,t,\nu(u))$ as in \cite{al83}.
\end{enumerate}
\end{remark}

Denote by $R_{\beta}$ the range of $\beta$ and for $s\in R_{\beta}$ define the right inverse
$\beta^{r}:R_{\beta}\to\RR$ of $\beta$ by
\begin{equation}
\beta^{r}(s) = 
	\begin{cases}
	\inf\{t\in\RR\,|\,\beta(t)=s\} & \text{if }s> 0,\\
	0 & \text{if }s=0,\\
	\sup\{t\in\RR\,|\,\beta(t)=s\} & \text{if }s < 0.
	\end{cases} \label{def:betar}
\end{equation}
That is, $\beta^r(s)$ is the closest $t$ to $0$ such that $\beta(t) = s$.
Since $\beta(0)=0$, note that $\beta^{r}$ is nondecreasing, nonnegative
on $R_{\beta}\cap\RR^+$ and nonpositive on
$R_{\beta}\cap \RR^-$. We can therefore extend $\beta^r$ as a function $\overline{R_\beta}
\to [-\infty,\infty]$. We then define $B:\overline{R_{\beta}}\to [0,\infty]$ by
\begin{equation}
B(z)=\int_{0}^{z} \zeta(\beta^{r}(s))\ud s. \label{def:B}
\end{equation}
The signs of $\zeta$ and $\beta^r$ ensure that $B$ is nonnegative 
on $\overline{R_\beta}$, nondecreasing on $\overline{R_{\beta}}\cap \RR^+$ and nonincreasing on
$\overline{R_{\beta}}\cap \RR^-$.
Moreover, since $\zeta$ and $\beta^r$ are non-decreasing, $B$ is convex on 
$\overline{R_\beta}$. This calls for extending $B$ as a function $\RR\to[0,+\infty]$ 
by setting $B=+\infty$ outside $\overline{R_\beta}$. This function is still nondecreasing 
on $\RR^+$ and nonincreasing on $\RR^-$.

Our notion of solution to \eqref{eq:model} is as follows.
\begin{definition}\label{def:solution}
Under Hypotheses \eqref{assumptions}, a \emph{solution} to \eqref{eq:model} is
a function $u$ satisfying
\begin{equation}\left\{
	\begin{gathered}
	u:\O\times(0,T)\to\RR \mbox{ is measurable}, \quad \zeta(u) \in \LSobo{p}{1}{p},\\
	B(\beta(u))\in \Leb{\infty}{1}, \quad \beta(u)\in C([0,T];L^2(\O)\weak),\\
	\partial_t\beta(u)\in \LSob{p'}{-1}{p'}, \quad \beta(u)(\cdot,0) = \beta(\uini)\mbox{ in $L^2(\O)$},\\
	\int_{0}^{T}\langle \partial_t\beta(u)(\cdot,t), v(\cdot,t)\rangle_{W^{-1,p'},W^{1,p}_0}\ud t\\
	+ \int_{0}^{T}\int_{\O}a(x,\nu(u(x,t)),\nabla\zeta(u)(x,t))\cdot\nabla v(x,t) \ud x\ud t \\
	= \int_{0}^{T}\langle f(\cdot,t),v(\cdot,t)\rangle_{W^{-1,p'},W^{1,p}_{0}} \ud t
	\quad \forall v\in L^p(0,T;W^{1,p}_0(\O)).
	\end{gathered}\right. \label{eq:solution}
\end{equation}
\end{definition}
Here $C([0,T];L^2(\O)\weak)$ denotes the space of continuous functions from $[0,T]$
into $L^{2}(\O)$, where the latter is equipped with the weak topology. This notion
of continuity for $\beta(u)$ can be understood as a natural consequence of the
integrability of $\beta(u)$ and the PDE itself. Fix $\varphi\in C^{\infty}_{c}(\O)$
and consider the map 
$\mathcal L_{\varphi}: [0,T]\to\RR$, $t\mapsto\langle \beta(u)(t), \varphi\rangle_{L^{2}(\O)}$.
One can show, using the PDE in the sense of distributions and the fact that
$\beta(u)\in\Leb{\infty}{2}$ (see Estimate \eqref{est:energy} below), 
that $\mathcal L_{\varphi}\in W^{1,1}(0,T)\subset C([0,T])$. From the density of $C^{\infty}_{c}(\O)$
in $L^{2}(\O)$ one deduces that for every $\varphi\in L^{2}(\O)$, $\mathcal L_{\varphi}\in C([0,T])$. 
That is, $\beta(u): [0,T]\to L^{2}(\O)\weak$ is continuous.

The main result of this paper is the following convergence theorem.
\begin{theorem}\label{th:main}
Let $(\betan,\zetan,\nun,\an,\fn,\uinin)_{n\in\NN}$ be a sequence converging to 
$(\beta,\zeta,\nu,a,f,\uini)$ in the following sense:
\begin{equation}\left\{
	\begin{aligned}
	&\mbox{$\betan$, $\zetan$ and $\nun$ converge locally uniformly on $\RR$ to 
	$\beta$, $\zeta$ and $\nu$, respectively;}\\
	&\mbox{for almost every $x\in\O$, $\an(x,\cdot,\cdot)\to a(x,\cdot,\cdot)$ locally
	uniformly on $\RR\times\RR^d$;}\\
	&\mbox{$\fn\to f$ in $\LSob{p'}{-1}{p'}$ and $\uinin\to\uini$ in $L^{2}(\O)$.}
	\end{aligned}\right. \label{convergences}
\end{equation}
Assume that $(\beta,\zeta,\nu,a,f,\uini)$ and $(\betan,\zetan,\nun,\an,\fn,\uinin)$
(for every $n\in\NN$) satisfy \eqref{assumptions} and \eqref{hyp:pbeta}, and that the constants $L_{\beta}$,
$L_{\zeta}$, $M_{1}$, $M_{2}$, $M_3$, $M_4$, $\underline{a}$, $\mu$ and the function $\overline{a}$  
are independent of $n$. Let $\un$ be a solution to \eqref{eq:perturbed}. 
Then there is a solution $u$ of \eqref{eq:model} such that, up to a subsequence,
\begin{equation}\left\{
	\begin{aligned}
	&\betan(\un) \to \beta(u) \quad \mbox{in $C([0,T]; L^{2}(\O)\weak)$,}\\
	&\nun(\un) \to \nu(u) \quad \mbox{in $C([0,T]; L^{2}(\O))$, and}\\
	&\zetan(\un)\weakto \zeta(u) \quad \mbox{weakly in }\LSobo{p}{1}{p}.
	\end{aligned}\right.\label{eq:mainconv}
\end{equation}
If in addition we assume that $a$ is strictly
monotone, that is, the inequality in \eqref{hyp:amono} with $\chi\not=\xi$ is strict, then
\begin{equation}
\zetan(\un) \to \zeta(u) \quad \mbox{strongly in $\LSobo{p}{1}{p}$.}\label{eq:strgrad}
\end{equation}
\end{theorem}

\begin{remark}
This theorem provides a stability result for any subclass of problem \eqref{eq:model}
for which uniqueness of the weak solution is known. This is indeed the case for simplified
versions of the Richards equation and Stefan problem; see Appendix \ref{app:unicite}.
In these settings, Theorem \ref{th:main} shows that the \emph{whole sequence} $\betan(\un)$ (respectively 
$\zetan(\un)$) converges uniformly-in-time to $\beta(u)$ (respectively $\zeta(u)$).
\end{remark}

\begin{remark}
Since $\betan$, $\zetan$ and $\nun$ are non-decreasing and 
$\beta$, $\zeta$ and $\nu$ are continuous, Dini's theorem shows that
we only need to assume that $\betan$, $\zetan$ and $\nun$ converge
pointwise. The Arzel\`a--Ascoli theorem can be used to arrive at the
same conclusion, since $\betan$, $\zetan$ and $\nun$ are uniformly Lipschitz continuous.
\end{remark}
\begin{remark} \label{rem:nu}
The local uniform convergence on $\RR$ of $\nun$ to $\nu$ holds if we assume that 
$\betan' \to \beta'$ almost everywhere on $\RR$, or 
$\zetan' \to \zeta'$ almost everywhere on $\RR$. Indeed, suppose that the latter pointwise
convergence holds. Since $(\betan')_{n\in\NN}$ is bounded by $L_\beta$, up to a
subsequence, $\betan' \weakto \chi$ weak-$\ast$ in $L^{\infty}(\RR)$ for
some bounded $\chi:\RR\to\RR$. Then as $n\to\infty$,
\begin{equation*}
\betan(s) = \int_{0}^{s}\betan'(q)\ud q \to \int_{0}^{s}\chi(q)\ud q.
\end{equation*}
But $\betan(s)\to\beta(s)$ for every $s\in\RR$, so it must be that $\chi=\beta'$
almost everywhere on $\RR$ and therefore that $\betan' \weakto \beta'$ weak-$\ast$ in 
$L^{\infty}(\RR)$. One can then pass to the limit as $n\to\infty$ in the definition of 
$\nun$, using dominated convergence on the sequence $(\zetan')_{n\in\NN}$, to obtain
the local uniform convergence towards $\nu$.
\end{remark}

\begin{remark}
Observe that in the case $1<p\le \frac{2d}{d+2}$ in \eqref{hyp:pbeta}, we 
do not need the strict monotonicity of each $\betan$; we only require that the limit
$\beta$ does not have any plateaux.
\end{remark}

As a by-product of this convergence result, we obtain existence for \eqref{eq:model}.

\begin{corollary}\label{cor:exist}
Under Hypotheses \eqref{assumptions} and \eqref{hyp:pbeta}, there exists a solution to \eqref{eq:model}.
\end{corollary}

\begin{proof}
Theorem \ref{th:main} shows that we only need to establish the existence of a solution
for perturbed problems \eqref{eq:model}. Upon replacing $\beta$ and $\zeta$ by $\beta+\delta\Id$
and $\zeta+\delta\Id$ for some small $\delta>0$, we can therefore assume that
\begin{equation*}
\beta'\geq\delta \quad\mbox{ and }\quad\zeta'\geq\delta \quad\mbox{on $\RR$}.
\end{equation*}
In particular, these perturbed $\beta$ and $\zeta$ are bi-Lipschitz homeomorphisms, 
and we define
\begin{equation}\label{def:a0}
a_0(x,s,\xi)=a(x,\nu(\beta^{-1}(s)),(\zeta\circ\beta^{-1})'(s)\xi),
\end{equation}
where for some $\rho>0$, $(\zeta\circ\beta^{-1})'(s)\in [\rho,\rho^{-1}]$
for all $s\in\RR$. The function $a_0$ satisfies
\eqref{hyp:acoerc}--\eqref{hyp:amono}. J.-L. Lions showed \cite{lio65} that there
exists a solution to
\begin{equation}\label{eq:v:nondeg}
\left\{
\begin{aligned} 
&\partial_t v-\dive \left( a_0(x,v,\nabla v) \right)= f &\quad &\mbox{in } \Omega\times(0,T),\\
&v(x,0) = \beta(\uini)(x) &\quad &\mbox{in } \Omega,\\
&v=0 &\quad &\mbox{on } \partial\Omega\times(0,T)
\end{aligned}\right.
\end{equation}
in the sense that $v\in L^p(0,T;W^{1,p}_0(\O))\cap C([0,T];L^2(\O))$,
$\partial_t v\in L^{p'}(0,T;W^{-1,p'}(\O))$, $v(\cdot,0)=\beta(\uini)$,
and the equation is satisfied against any test function in
$L^p(0,T;W^{1,p}_0(\O))$.

We then set $v=\beta(u)$. Then $\zeta(u)=(\zeta\circ\beta^{-1})(v)$ with
$\zeta\circ\beta^{-1}$ Lipschitz continuous,
and thus $\zeta(u)\in L^p(0,T;W^{1,p}_0(\O))$. We have $\beta(u)=v\in
C([0,T];L^2(\O))$, $\beta(u)(\cdot,0)=v(\cdot,0)=\beta(\uini)$, and
$\partial_t \beta(u)=\partial_t v\in L^{p'}(0,T;W^{-1,p'}(\O))$.
The definition \eqref{def:a0} of $a_0$ shows that
\begin{equation*}
a_0(x,v,\nabla v)=a(x,\nu(\beta^{-1}(v)),(\zeta\circ\beta^{-1})'(v)\nabla v)=
a(x,\nu(u),\nabla \zeta(u))
\end{equation*}
and thus the integral equation in \eqref{eq:solution} follows from
writing the equation \eqref{eq:v:nondeg} against test functions in $L^p(0,T;W^{1,p}_0(\O))$.
Finally, since $B\circ \beta$ grows quadratically (see \eqref{eq:growthB} below)
and $u=\beta^{-1}(v)\in C([0,T];L^2(\O))$, we have $B(\beta(u))\in L^\infty(0,T;L^1(\O))$.
Thus $u$ is a solution to \eqref{eq:model}. \end{proof}


\section{A maximal monotone operator viewpoint}\label{sec:maxmonop}

This section demonstrates that our setting covers problems defined by 
sublinear maximal monotone operators. We begin with a lemma.

\begin{lemma}[Maximal monotone operator]\label{lem:maxmongr}
Let $\mmope:\RR\to \mathcal P(\RR)$ be a multi-valued operator. Then the following
are equivalent:
\begin{enumerate}[(i)]
\item\label{equiv:it1} $\mmope$ is a maximal monotone operator with domain $\RR$,
$0\in \mmope(0)$ and $\mmope$ is sublinear in the sense that there exist 
$T_1, T_2\geq 0$ such that, for all $x\in \RR$ and all $y\in \mmope(x)$, $|y|\leq T_1|x|+T_2$;
\item\label{equiv:it2} There exist $\zeta$ and $\beta$ satisfying
\eqref{hyp:zeta} and \eqref{hyp:beta} such that the graph of $\mmope$
is given by $\graph(\mmope) = \{(\zeta(s),\beta(s)),s\in\RR\}$.
\end{enumerate}
\end{lemma}

\begin{proof} 
\emph{(\ref{equiv:it2})$\Rightarrow$(\ref{equiv:it1}).}
Clearly $0=(\zeta(0),\beta(0))\in \mmope(0)$.
The monotonicity of $\mmope$ follows from the fact that $\zeta$ and $\beta$ are nondecreasing.
We prove that $\mmope$ is maximal, that is if
$x,y$ satisfy $(\zeta(s)-x)(\beta(s) - y)\ge 0$ for all $s\in\RR$ then
$(x,y)\in \graph(\mmope)$. 
By \eqref{hyp:beta} and \eqref{hyp:zeta} the mapping $\beta+\zeta:\RR\to\RR$ is surjective,
so there exists $w\in\RR$ such that 
\begin{equation}\label{rem:xy}
\beta(w)+\zeta(w)=x+y.
\end{equation}
Then $\zeta(w)-x=y-\beta(w)$ and therefore
$0\le (\zeta(w)-x)(\beta(w) - y)=-(\beta(w)-y)^2$. This implies $\beta(w)=y$ and,
combined with \eqref{rem:xy}, $\zeta(w)=x$. Hence $(x,y)\in \graph(\mmope)$.
The sub-linearity of $T$ follows from 
$|\beta(w)| \le L_\beta |w| \le L_\beta ( |\zeta(w)|+M_2 )/M_1$.

\emph{(\ref{equiv:it1})$\Rightarrow$(\ref{equiv:it2}).} Recall that the resolvent $\mathcal R(\mmope)=
(\Id+\mmope)^{-1}$ of the maximal monotone operator $\mmope$ is a single-valued function $\RR\to\RR$
that is nondecreasing and Lipschitz continuous with Lipschitz constant $1$. Set $\zeta=\mathcal R(\mmope)$ and
$\beta=\Id-\zeta$. These functions are nondecreasing and Lipschitz continuous with constant $1$.
By definition of the resolvent,
\begin{equation*}
(x,y)\in\graph(\mmope)\Leftrightarrow (x,x+y)\in\graph(\Id+\mmope)
\Leftrightarrow (x+y,x)\in \graph(\zeta)\Leftrightarrow x=\zeta(x+y).
\end{equation*}
Since $\beta=\Id-\zeta$, setting $s=x+y$ shows that $(x,y)\in\graph(\mmope)$ is equivalent to $(x,y)=(\zeta(s),\beta(s))$.
Since $0\in T(0)$ this gives $\beta(0)=\zeta(0)=0$. Finally, the existence of $M_1$ and $M_2$
in \eqref{hyp:zeta} follows from the sublinearity of $\mmope$. If $(x,y)\in \graph(\mmope)$ then
$|y|\le T_1|x|+T_2$ and $x=\zeta(x+y)$, which gives
$|x+y| \le ((1+T_1) |\zeta(x+y)| + T_2)$.
\end{proof}

Using this lemma, we recast \eqref{eq:model} as
\begin{equation}\left\{
\begin{aligned} 
&\partial_t \mmope(z)-\dive \left( a(x,\nu(z+\mmope(z)),\nabla z) \right)= f &\quad &\mbox{in } \Omega\times(0,T),\\
& \mmope(z)(\cdot,0) = b^{\rm ini} &\quad &\mbox{in } \Omega,\\
& z=0 &\quad &\mbox{on } \partial\Omega\times(0,T).
\end{aligned}\right. \label{eq:modelmonot} \tag{PM}
\end{equation}
Hypotheses \eqref{hyp:beta} and \eqref{hyp:zeta} translate into
\bhyp{hyp:H}
\mbox{$\mmope$ is a maximal monotone operator with domain $\RR$,
$0\in \mmope(0)$}\\
\mbox{and $\mmope$ is sublinear in the sense that there exist $T_1,T_2\ge 0$ such that,}\\
\mbox{for all $x\in \RR$ and all $y\in \mmope(x)$, $|y|\le T_1|x|+T_2$.}
\ehyp
Hypothesis \eqref{hyp:pbeta} becomes
\begin{equation}\left\{
	\begin{aligned}
	&\mbox{\rm (I)}&&\mbox{$p\geq 2$,}\\
	&\mbox{or}\\
	&\mbox{\rm (II)}&&\mbox{$\frac{2d}{d+2}<p<2$ and there are positive constants 
	$T_3$, $T_4$ such that}\\
	&&&\mbox{for all $(x,y)\in \graph(\mmope)$, $|y|\geq T_{3}|x| - T_{4}$,}\\
	&\mbox{or}\\
	&\mbox{\rm (III)}&&\mbox{$1< p\leq \frac{2d}{d+2}$, there are positive constants
	$T_3$, $T_4$ such that}\\
	&&&\mbox{for all $(x,y)\in \graph(\mmope)$, $|y|\geq T_{3}|x| - T_{4}$, and $\mmope$ is
	strictly monotone.}
	\end{aligned}\right.\label{hyp:pbetamonot}
\end{equation}

In \eqref{eq:modelmonot}, $\nu$ is defined as the anti-derivative of $\zeta'\beta'$, where
$\zeta = \mathcal{R}(\mmope)$ and $\beta = \Id - \zeta$.
The reciprocal $\mmope^{-1}$ of $\mmope$ is itself a maximal monotone operator, and the
function $ \zeta(\beta^{r}(s)) $ in \eqref{def:B} can be computed in terms of $\mmope^{-1}$:
$\zeta(\beta^r(s))= \inf \mmope^{-1}(s)$ if $s> 0$, $\zeta(\beta^r(0))=0$,
and $\zeta(\beta^r(s))=\sup \mmope^{-1}(s)$ if $s<0$.
We then see that, for all $s$ in the domain of $\mmope^{-1}$,
$\mmope^{-1}(s)$ is the convex sub-differential $\partial B(s)$ of $B$ at $s$. 

\begin{definition}\label{def:solutionmonot}
Under Hypotheses \eqref{hyp:H} and \eqref{hyp:acoerc}--\eqref{hyp:uinif}, 
take a measurable function $b^{\rm ini}$ satisfying
$b^{\rm ini}(x)\in \mmope(\uini(x))$ for a.e. $x\in\O$.
A \emph{solution} to \eqref{eq:modelmonot} is a pair of functions $(z,b)$ satisfying
\begin{equation*}\left\{
	\begin{gathered}
	z \in \LSobo{p}{1}{p}\,,\;b(x,t)\in \mmope(z(x,t))\hbox{ for a.e. } (x,t)\in\O\times(0,T), \\
	B(b)\in \Leb{\infty}{1}, \quad b\in C([0,T];L^2(\O)\weak),\\
	\partial_t b\in \LSob{p'}{-1}{p'}, \quad b(\cdot,0) = b^{\rm ini}\mbox{ in $L^2(\O)$},\\
	\int_{0}^{T}\langle \partial_t b(\cdot,t), v(\cdot,t)\rangle_{W^{-1,p'},W^{1,p}_0}\ud t
	+ \int_{0}^{T}\int_{\O}a(x,\nu((b+z)(x,t)),\nabla z(x,t))\cdot\nabla v(x,t) \ud x\ud t \\
	= \int_{0}^{T}\langle f(\cdot,t),v(\cdot,t)\rangle_{W^{-1,p'},W^{1,p}_{0}} \ud t
	\quad \forall v\in L^p(0,T;W^{1,p}_0(\O)).
	\end{gathered}\right. \label{eq:solutionmonot}
\end{equation*}
\end{definition}

\begin{remark}The sublinearity of $\mmope$ ensures that $b^{\rm ini}\in L^2(\O)$
since $\uini\in L^2(\O)$.
\end{remark}

The following theorem is an immediate consequence of Theorem \ref{th:main} and
Corollary \ref{cor:exist}. We simply take
$u = b+z$, which implies $\beta(u)=b$ and $\zeta(u)=z$ since $(z,b)\in \graph(\mmope)$.

\begin{theorem}\label{th:mainmonot}
Under Hypotheses \eqref{hyp:acoerc}--\eqref{hyp:uinif}, \eqref{hyp:H} and
\eqref{hyp:pbetamonot}, \eqref{eq:modelmonot} has at least one solution.

Moreover, let $(\mmope_n,\an,\fn,\uinin)_{n\in\NN}$ be a sequence that converges to 
$(\mmope,a,f,\uini)$ in the following sense:
\begin{equation*}\left\{
	\begin{aligned}
	&\mbox{$ \mathcal{R}(\mmope_n)$ and $\nun$ converge locally uniformly on $\RR$ to $ \mathcal{R}(\mmope)$ and $\nu$ respectively;}\\
	&\mbox{for almost every $x\in\O$, $\an(x,\cdot,\cdot)\to a(x,\cdot,\cdot)$ locally
	uniformly on $\RR\times\RR^d$;}\\
	&\mbox{$\fn\to f$ in $\LSob{p'}{-1}{p'}$ and $\uinin\to \uini$ in $L^{2}(\O)$.}
	\end{aligned}\right. \label{convergencesmonot}
\end{equation*}
Assume that $(\mmope,a,f,\uini)$ and $(\mmope_n,\an,\fn,\uinin)$
(for every $n\in\NN$) satisfy \eqref{hyp:H} and \eqref{hyp:acoerc}--\eqref{hyp:uinif}, 
and that the constants $T_1$, $T_2$, $T_3$, $T_4$, $\underline{a}$, $\mu$ and the function $\overline{a}$  
are independent of $n$. Let $(z_n,b_n)$ be a solution of \eqref{eq:modelmonot} with
$(\mmope,a,f,\uini)$ replaced with $(\mmope_n,\an,\fn,\uinin)$. Then 
there is a solution $(z,b)$ of \eqref{eq:modelmonot} such that, up to a subsequence,
\begin{equation*}\left\{
	\begin{aligned}
	&b_n \to b \quad \mbox{in $C([0,T]; L^{2}(\O)\weak)$,}\\
	&\nun(b_n + z_n) \to \nu(b+z) \quad \mbox{in $C([0,T]; L^{2}(\O))$, and}\\
	&z_n\weakto z\quad \mbox{weakly in }\LSobo{p}{1}{p}.
	\end{aligned}\right.\label{eq:mainconvmonot}
\end{equation*}
If in addition we assume that $a$ is strictly
monotone, that is, the inequality in \eqref{hyp:amono} with $\chi\not=\xi$
is strict, then $z_n \to z$ strongly in $\LSobo{p}{1}{p}$.
\end{theorem}

\begin{remark}
Blanchard and Porretta \cite{bp05} prove the $\Leb{\infty}{1}$-stability of 
renormalised solutions to
\begin{equation*}\left\{
\begin{aligned} 
&\partial_t \mmope(u)-\dive \left( a(x, u,\nabla u) \right) + \dive\left(\Phi(u)\right)
= f &\quad &\mbox{in } \Omega\times(0,T),\\
&\mmope(u)(x,0) = b_{0}(x) &\quad &\mbox{in } \Omega,\\
&u=0 &\quad &\mbox{on } \partial\Omega\times(0,T),
\end{aligned}\right.
\end{equation*}
with $f\in L^{1}(\O\times(0,T))$ and $b_{0}\in L^{1}(\O)$. They
assume that $\mmope$ is a maximal monotone graph on $\RR$, $\mmope^{-1}\in C(\RR)$ and $a$ 
is a Leray-Lions operator. Although the continuity assumption on $\mmope^{-1}$ -- which prevents $\mmope$ from having
plateaux -- is not required for the stability result, it is necessary in their existence
theorem for identifying
$u$ as the pointwise limit of solutions to regularised problems using compactness arguments.
If $p>\frac{2d}{d+2}$ we overcome this assumption on $\mmope$ in the variational
setting by using monotonicity and compensated compactness arguments; see
Section \ref{ssec:step2}. Indeed, it may be interesting to determine whether
similar arguments may be used in the setting of renormalised solutions
in \cite{bp05}.
If $p$ is `too small' -- that is, in case (III) of Hypothesis \eqref{hyp:pbetamonot} --
we must also assume that $\mmope$ (respectively $\beta$ outside the present section)
does not have any plateaux, but we still identify weak limits by montony
and compensated compactness arguments rather than by pointwise convergence.
\end{remark}

\section{Preliminaries} \label{sec:prelims}

\subsection{Properties of $B$} \label{ssec:B}

We recall here two lemmas proved in \cite{dey14}. Lemma \ref{lem:techineq}
states some properties of the functions $\nu$ and $B$. The uniform convexity property 
\eqref{eq:Bunifconvex} plays a critical role in our proof of the uniform
temporal convergence of $\nun(\un)$. Lemma \ref{lem:integparts} brings together two
identities --- an integration-by-parts formula and an energy
equality --- and some continuity properties of the solution. Although the
integration-by-parts formula \eqref{eq:integparts}
apparently follows from the formal relation
$\zeta(u) \partial_t (\beta(u))=\zeta(u)\beta'(u)\partial_t u=(B\circ\beta)'(u)\partial_tu
=\partial_t (B(\beta(u))$, its rigorous justification is quite technical, owing to the
lack of regularity of $u$.

\begin{lemma}\label{lem:techineq}
Assume \eqref{assumptions}. Then for every $a,b\in\RR$,
\begin{subequations}
\begin{equation}
|\nu(a) - \nu(b)| \leq L_{\beta}|\zeta(a) - \zeta(b)|, \quad \mbox{and}\label{eq:nubetatrans}
\end{equation}
\begin{equation} \label{eq:nubetatrans2}
[\nu(a) - \nu(b)]^{2} \leq L_{\beta}L_{\zeta}[\zeta(a)-\zeta(b)][\beta(a) - \beta(b)].
\end{equation}
The functions $B:\overline{R_{\beta}}\to[0,\infty]$ and
$B\circ\beta : \RR \to [0,\infty)$ are continuous, and for all $s\in\RR$,
\begin{equation} \label{eq:relBbeta}
B(\beta(s)) = \int_{0}^{s}\zeta(q)\beta'(q)\ud q.
\end{equation}
There are positive constants $K_{1}, K_{2}$ and $K_{3}$, depending only upon
$L_{\beta}$, $L_{\zeta}$ and the constants $M_{1},M_2$ in \eqref{hyp:zeta},
such that for all $s\in\RR$,
\begin{equation}
K_{1}\beta(s)^{2} - K_{2}\leq B(\beta(s))\leq K_{3}s^{2}. \label{eq:growthB}
\end{equation}
Finally, for every $a,b\in\RR$,
\begin{equation} \label{eq:Bunifconvex}
[\nu(a) - \nu(b)]^{2}\leq 4L_{\beta}L_{\zeta}\left[B(\beta(a))+ B(\beta(b))-2 B\left(\frac{\beta(a)+\beta(b)}{2}\right)\right].
\end{equation}
\end{subequations}
\end{lemma}
Before stating the next lemma, a few remarks on notation are necessary.
The discussion following Definition \ref{def:solution} concerning the continuity
of $\beta(u)$ neglects a subtlety that one must account for in order to give meaning
to the convergences \eqref{eq:mainconv}. Indeed, by the statement 
$\beta(u)\in C([0,T]; L^{2}(\O)\weak)$, one understands that the mapping 
$(x,t)\mapsto\beta(u(x,t))$ is equal almost everywhere on $\O\times(0,T)$ to a function
$Z$ that is continuous as a map from $[0,T]$ to $L^{2}(\O)\weak$. 
We henceforth write $\ol{\beta(u)}$ for $Z$; similarly $\ol{\nu(u)}$ for the continuous 
in time almost-everywhere representative of $(x,t)\mapsto\nu(u(x,t))$ (see part (ii) of the following lemma).
This distinction is essential in the present context, where we are frequently
concerned with the values of these functions at a particular point in time. The composition
$\beta(u(\cdot,\cdot))$ is only defined up to null sets in $\O\times(0,T)$, so for a
particular $t\in[0,T]$ the expression $\beta(u(\cdot,t))$ is ill-defined. The expression
$\ol{\beta(u)}(\cdot, t)$ is, however, well-defined, and we take care to use
the notation $\beta(u)$ (without the bar) only when this quantity is used in an average sense.
Nonetheless, for the sake of clarity Theorem \ref{th:main} is stated without
this distinction.

\begin{lemma}\label{lem:integparts}
Let \eqref{assumptions} hold.
\begin{enumerate}[(i)]
\item If $v$ is a measurable function on $\O\times(0,T)$ such that
$\zeta(v)\in L^p(0,T;W^{1,p}_0(\O))$, $B(\beta(v))\in L^\infty(0,T;L^1(\O))$,
$\beta(v)\in C([0,T];L^2(\O)\weak)$ and $\partial_t \beta(v)\in L^{p'}(0,T;W^{-1,p'}(\O))$ then
the mapping $[0,T]\ni t \mapsto \int_{\O}B(\ol{\beta(v)}(x,t))\ud x \in [0,\infty)$
is continuous and bounded, and for all $T_0\in [0,T]$,
\begin{equation}
\int_{0}^{T_{0}}\langle \partial_{t}\beta(v)(\cdot,t),\zeta(v(\cdot,t))\rangle_{W^{-1,p'},W^{1,p}_{0}}\ud t
= \int_{\O}B(\ol{\beta(v)}(x,T_{0}))\ud x - \int_{\O}B(\ol{\beta(v)}(x,0))\ud x. \label{eq:integparts}
\end{equation}
\item If $u$ is a solution to \eqref{eq:model} then for all $T_0\in [0,T]$,
\begin{multline} 
\label{eq:energy}
\int_{\O}B(\ol{\beta(u)}(x,T_{0}))\ud x + \int_{0}^{T_{0}}\int_{\O}a(x,\nu(u),\nabla\zeta(u))\cdot\nabla\zeta(u)\ud x\ud t\\
= \int_{\O}B(\beta(\uini(x)))\ud x + \int_{0}^{T_{0}}\langle f(\cdot,t),\zeta(u)(\cdot,t)\rangle_{W^{-1,p'},W^{1,p}_{0}} \ud t
\end{multline}
and the function $\nu(u)$ is continuous from $[0,T]$ into $L^{2}(\O)$.
\end{enumerate}
\end{lemma}

Since $B$ plays such a critical role to our main result, we highlight
its stability properties in the following lemma.
\begin{lemma}\label{lem:convBn}
Assume \eqref{assumptions} and \eqref{convergences}. 
Define $B_n: \overline{R_{\betan}}\to [0,\infty]$ from $\zetan$, $\betan$ analogously to \eqref{def:B}, and
extend $B_n$ to $\RR$ by setting $B_n=+\infty$ outside $\overline{R_{\betan}}$.
Then
\begin{enumerate}[(i)]
\item\label{convBn:it1} $B$ and all $(B_n)_{n\in\NN}$ are convex lower semi-continuous
on $\RR$;
\item\label{convBn:it2} $\Bn\circ\betan\to B\circ\beta$ locally uniformly on $\RR$ as $n\to\infty$;
\item\label{convBn:it3} For any $z\in\RR$ and any sequence $(z_n)_{n\in\NN}$ that
converges to $z$, $B(z)\le \liminf_{n\to\infty} B_n(z_n)$.
\end{enumerate}
\end{lemma}
\begin{proof}
(\ref{convBn:it1}) The convexity has already been noted. Since $B$ and $B_n$
are continuous on $\overline{R_\beta}$ and $\overline{R_{\betan}}$
respectively by Lemma \ref{lem:techineq},
their extension by $+\infty$ outside their initial domain ensures their
lower semi-continuity.

(\ref{convBn:it2}) Let $M>0$. By \eqref{eq:relBbeta} applied to $B_n$, $B_n(\betan)$ is Lipschitz continuous on $[-M,M]$ with
Lipschitz constant $L_\beta\sup_{|s|\le M}|\zetan(s)|$. This quantity is bounded 
with respect to $n$ since $(\zetan)_{n\in\NN}$
converges uniformly on $[-M,M]$. Hence, the local uniform convergence of $(B_n(\betan))_{n\in\NN}$ follows
from the Arzel\`a--Ascoli theorem if we can prove that $B_n(\betan)\to B(\beta)$ pointwise.
The reasoning in Remark \ref{rem:nu} shows that $\betan'\weakto\beta'$ weak-$\ast$ in $L^\infty(\RR)$.
Hence, for any $s\in\RR$, since $\zetan\to\zeta$ uniformly on $[0,s]$,
\begin{equation*}
B_n(\betan(s))=\int_0^s \zetan(q)\betan'(q)\ud q \to 
\int_0^s \zeta(q)\beta'(q)\ud q=B(\beta(s))\mbox{ as $n\to\infty$},
\end{equation*}
and the proof of (\ref{convBn:it2}) is complete.

(\ref{convBn:it3}) Without loss of generality, we can assume that $(B_n(z_n))_{n\in\NN}$
converges in $[0,\infty]$, otherwise we extract a subsequence that converges to
the inferior limit. We study four distinct cases.

\emph{Case A: $z_n\not\in \overline{R_{\betan}}$ for an infinite number of $n$.}
Then the corresponding $B_n(z_n)$ are equal to $+\infty$ and therefore
$\lim_{n\to\infty} B_n(z_n)=+\infty\ge B(z)$.

\emph{Case B: $z_n\in \overline{R_{\betan}}$ for $n$ large, and $z\not\in \overline{R_\beta}$.}
Assume that $z>\sup R_\beta$ (the case $z<\inf R_\beta$ is similar). Take
$Z\in (\sup R_\beta,z)\subset (0,\infty)$. For $n$ sufficiently large,
$z_n>Z$ and $z_n\in \overline{R_{\betan}}$. Then use the definition \eqref{def:B}
of $B_n$, Hypothesis \eqref{hyp:zeta} and the fact that $\betan^r$ is nondecreasing
to see that
\begin{multline}\label{eq:nicer}
B_n(z_n)= \int_0^{z_n}\zetan(\betan^r(s))\ud s\ge
\int_Z^{z_n}\zetan(\betan^r(s))\ud s\\
\ge \int_Z^{z_n} (M_1\betan^r(s)-M_2)\ud s
\ge (z_n-Z)(M_1\betan^r(Z)-M_2).
\end{multline}
We prove by contradiction that $(\betan^r(Z))_{n\in\NN}$ is not bounded.
Otherwise, upon extraction of a subsequence it converges to some $m\in\RR$. Then,
by local uniform convergence of $\betan$, $Z=\betan(\betan^r(Z))
\to \beta(m)\in R_{\beta}$. But $Z>\sup R_\beta$, which is a contradiction.
Hence, $\betan^r(Z)\to+\infty$ as $n\to\infty$. Since $z_n-Z\to z-Z>0$,
passing to the limit in \eqref{eq:nicer} gives $\lim_{n\to\infty} B_n(z_n)=+\infty\ge B(z)$.

\emph{Case C: $z_n\in \overline{R_{\betan}}$ for $n$ large, $z\in \overline{R_\beta}$
and $(\betan^r(z_n))_{n\in\NN}$ is bounded in $\RR$.}
Let $s_n=\betan^r(z_n)$, which gives $z_n=\betan(s_n)$.
Since $(s_n)_{n\in\NN}$ is bounded, up to extraction of a subsequence we have $s_n\to s\in\RR$ and thus, 
by (\ref{convBn:it2}), $B_n(z_n)=B_n\circ \betan(s_n)
\to B\circ\beta(s)$. The local uniform convergence of $(\betan)_{n\in\NN}$
gives $z_n=\betan(s_n)\to \beta(s)$, which means that $\beta(s)=z$. Hence
$B_n(z_n)\to B(\beta(s))=B(z)$ and the proof is complete.

\emph{Case D: $z_n\in \overline{R_{\betan}}$ for $n$ large, $z\in \overline{R_\beta}$
and $(\betan^r(z_n))_{n\in\NN}$ is unbounded.}
Again, let $s_n=\betan^r(z_n)\in [-\infty,+\infty]$. The function $B_n$ is continuous 
(with values in $[0,+\infty]$) at the endpoints of $R_{\betan}$. Since these endpoints correspond to $\lim_{s\to \pm\infty}
\betan(s)$, applying the monotone convergence theorem to \eqref{eq:relBbeta} then
shows that this formula also holds if $s=\pm\infty$. Hence, for any $n$,
\begin{equation*}
B_n(z_n)=B_n(\betan(s_n))=\int_0^{s_n}\zetan(q)\betan'(q)\ud q.
\end{equation*}
The sequence $(s_n)_{n\in\NN}$ contains a subsequence that goes to $\pm\infty$.
Say, without explicitly denoting the subsequence, that $s_n \to +\infty$ 
(the case $s_n\to -\infty$ is similar). Let $M\ge 0$
and for $n$ sufficiently large, since $\zetan\ge 0$ on $\RR^+$ and $\betan'\ge 0$,
write
\begin{equation*}
B_n(z_n)=\int_0^{s_n}\zetan(q)\betan'(q)\ud q\ge \int_\RR \ch_{[0,M]}(q) \zetan(q)\betan'(q)\ud q.
\end{equation*}
By the reasoning in Remark \ref{rem:nu}, $\betan'\to \beta'$ in $L^\infty(\RR)$ weak-$\ast$.
Since $\zetan\to \zeta$ uniformly on $[0,M]$, we can conclude that
\begin{equation*}
\lim_{n\to\infty}B_n(z_n)\ge \int_\RR \ch_{[0,M]}(q)\zeta(q)\beta'(q)\ud q.
\end{equation*}
Take the limit inferior as $M\to\infty$ using Fatou's lemma to deduce that
\begin{equation*}
\lim_{n\to\infty}B_n(z_n)\ge \int_0^\infty\zeta(q)\beta'(q)\ud q.
\end{equation*}
Since $z\ge 0$ (because for $n$ large enough, each $z_n=\betan(s_n)$ is nonnegative),
$s=\beta^r(z)\in [0,\infty]$ and thus
\begin{equation*}
\lim_{n\to\infty}B_n(z_n)\ge \int_0^s\zeta(q)\beta'(q)\ud q.
\end{equation*}
We already saw that \eqref{eq:relBbeta} is valid for any $s\in [-\infty,\infty]$,
and we infer that $\lim_{n\to\infty}B_n(z_n)\ge B(\beta(s))=B(z)$ as required. 
\end{proof}

\subsection{Estimates} \label{ssec:est}

The results of the previous section enable energy estimates,
the subject of our next lemma. Note that none of the estimates we prove in this
section require Hypothesis \eqref{hyp:pbeta}.


\begin{lemma}\label{lem:energyest}
Let $(\betan,\zetan,\nun,\an,\fn,\uinin)_{n\in\NN}$ be a sequence of data that
satisfies the hypotheses of Theorem \ref{th:main}, and let $\un$ be a solution
to \eqref{eq:perturbed}.
Then there exists $\ctel{energy}>0$ independent of $n$
such that
the following quantities are bounded above by $\cter{energy}$:
\begin{equation}\label{est:energy}
	\begin{aligned}
	&\sup_{t\in [0,T]}\norm{\Bn(\ol{\betan(\un)}(\cdot,t))}_{L^1(\O)},
	&\norm{\zetan(\un)}_{\LSobo{p}{1}{p}},\\
	&\sup_{t\in [0,T]}\norm{\ol{\betan(\un)}(\cdot,t)}_{L^2(\O)},
	&\norm{\partial_{t}\betan(\un)}_{\LSob{p'}{-1}{p'}}.	
	\end{aligned}
\end{equation}
\end{lemma}
\begin{proof}
By hypothesis, $(\betan,\zetan,\nun,\an,\fn,\uinin)$ (for every $n\in\NN$)
satisfies an identity analogous to \eqref{eq:energy}. From this identity, 
the quadratic growth \eqref{eq:growthB} of $\Bn\circ\betan$, 
the uniform coercivity of $(a_n)_{n\in\NN}$ and Young's inequality,
\begin{multline}\label{est:11}
\int_{\O}\Bn(\ol{\betan(\un)}(x,T_{0}))\ud x + \underline{a}\int_{0}^{T_{0}}\int_{\O}|\nabla\zetan(\un)|^{p}\ud x \ud t \\
\leq K_{3}\norm{\uinin}_{L^{2}(\O)}^{2} + \frac{\underline{a}}{2}\norm{\nabla\zetan(\un)}_{L^p(0,T_0;L^p(\O)^d)}^{p}
+ \frac{1}{p'}\left(\frac{2}{\underline{a}p}\right)^{p'/p}\norm{\fn}_{L^{p'}(0,T_0;W^{-1,p'}(\O))}^{p'}.
\end{multline}
Taking $T_0=T$ shows that
\begin{align*}
\norm{\nabla\zetan(\un)}_{\Leb[d]{p}{p}}^{p}
\leq \frac{2}{\underline{a}}
\left( K_{3}\norm{\uinin}_{L^{2}(\O)}^{2} + \frac{1}{p'}\left(\frac{2}{\underline{a}p}\right)^{p'/p}\norm{\fn}_{\LSob{p'}{-1}{p'}}^{p'}\right).
\end{align*}
With the assumed convergence properties of $(\uinin)_{n\in\NN}$ and $(f_n)_{n\in\NN}$,
substituting the previous inequality into \eqref{est:11} gives the first two estimates in \eqref{est:energy}.
The estimate on $(\ol{\betan(\un)})_{n\in\NN}$ follows from 
that on $(\Bn(\ol{\betan(\un)}))_{n\in\NN}$ and \eqref{eq:growthB}. To prove
the estimate on $\partial_t \betan(\un)$, let
$v\in\LSobo{p}{1}{p}$ and deduce from \eqref{eq:solution} that
\begin{multline*}
\left|\int_{0}^{T}\left\langle \partial_{t}\betan(\un)(\cdot, t), v(\cdot, t)\right\rangle_{W^{-1,p'},W^{1,p}_{0}}\ud t\right|\\
\leq \norm{v}_{\LSobo{p}{1}{p}}\left(\norm{\overline{a}}_{L^{p'}(\O)} 
+ \mu\norm{\nabla\zetan(\un)}_{\Leb[d]{p}{p}}^{p-1}
+ \norm{\fn}_{\LSob{p'}{-1}{p'}}\right).
\end{multline*}
Take the supremum over $v$ in the unit ball of $\LSobo{p}{1}{p}$ and use the bound
on $(\zetan(\un))_{n\in\NN}$ in $\LSobo{p}{1}{p}$ to complete the proof.
\end{proof}

The following lemma, applied to $F_n=\zetan$, $G_n=\betan$ and $\un$ the solution to
\eqref{eq:perturbed}, provides us with crucial estimates of the time translates of $\nun(\un)$. 
Nevertheless we state it in a generic setting, as it will also be applied with different functions.

\begin{lemma}\label{lem:nutrans}
For every $n\in\NN$, let $F_n:\RR\to\RR$ and $G_n:\RR\to\RR$ be nondecreasing and
Lipschitz continuous, uniformly with respect to $n$. Suppose also that $F_n(0)=0$. 
Define $H_n(s):=\int_0^s F_n'(q)G_n'(q)\ud q$.
Take $p\geq1$ and $(\un)_{n\in\NN}$ a sequence of measurable functions
on $\O\times(0,T)$ such that $(F_n(u_n))_{n\in\NN}$ is bounded in $L^p(0,T;W^{1,p}_0(\O))$,
$(G_n(u_n))_{n\in\NN}$ is bounded in $L^\infty(0,T;L^2(\O))$ and
$(\partial_t (G_n(u_n)))_{n\in\NN}$ is bounded in $L^{p'}(0,T;W^{-1,p'}(\O))$.
Then there exists $\ctel{ttrans}>0$ independent of $n$ such that, for all $2\le r<\infty$
and all $\tau>0$,
\begin{equation} \label{est:nutrans}
\norm{H_n(\un)(\cdot, \cdot + \tau) - H_n(\un)}_{L^r(\RR;L^2(\O)))}
\leq \cter{ttrans}\tau^{1/r},
\end{equation}
where $H_n(\un)$ is extended by zero outside $\O\times(0,T)$.
\end{lemma}

\begin{proof}
Denote by $L_{F}$ and $L_{G}$ the uniform Lipschitz constants of $(F_n)_{n\in\NN}$
and $(G_n)_{n\in\NN}$, respectively.
We introduce the truncations $\trunc_{k}:\RR\to\RR$ at level $k>0$, defined by
$\trunc_{k}(s):=\max(-k,\min(s,k))$, and the functions
\begin{equation*}
F_n^{k}(s) := \trunc_{k}(F_n(s)) \quad \mbox{and} \quad H_n^{k}(s):=\int_{0}^{s}(F_n^{k})'(q)G_n'(q)\ud q.
\end{equation*}
Then $F_n^{k}(\un), H_n^{k}(\un)\in\Leb{2}{2}$, the latter coming from
\begin{equation}
|H_n^{k}(\un)|\leq  L_F|G_n(\un)|.\label{eq:nudom}
\end{equation}
Now let $\tau\in(0,T)$. Inequality \eqref{eq:nubetatrans2} with $(F_n^k,G_n,H_n^k)$
in place of $(\zeta,\beta,\nu)$ yields
\begin{multline*} 
\int_{0}^{T-\tau}\int_{\O}\left(H_n^{k}(\un)(x,t+\tau) - H_n^{k}(\un)(x,t)\right)^{2}\ud x \ud t \\
\leq L_{F}L_{G}\int_{0}^{T-\tau}\int_{\O}\left(G_n(\un)(x,t+\tau) - G_n(\un)(x,t)\right)
\left(F_n^{k}(\un)(x,t+\tau) - F_n^{k}(\un)(x,t)\right)\ud x \ud t\\
= L_{F}L_{G}\int_{0}^{T-\tau}\left\langle G_n(\un)(\cdot,t+\tau) - G_n(\un)(\cdot,t),
F_n^{k}(\un)(\cdot,t+\tau) - F_n^{k}(\un)(\cdot,t)\right\rangle_{W^{-1,p'},W^{1,p}_{0}}\ud t\\
= L_{F}L_{G}\int_{0}^{T-\tau}\left\langle\int_{t}^{t+\tau}\partial_{t}G_n(\un)(\cdot,s)\ud s,
F_n^{k}(\un)(\cdot,t+\tau) - F_n^{k}(\un)(\cdot,t)\right\rangle_{W^{-1,p'},W^{1,p}_{0}} \ud t\\
= L_{F}L_{G}\int_{0}^{T-\tau}\int_{t}^{t+\tau}
\left\langle \partial_{t}G_n(\un)(\cdot,s), F_n^{k}(\un)(\cdot,t+\tau) - F_n^{k}(\un)(\cdot,t)
\right\rangle_{W^{-1,p'},W^{1,p}_{0}}\ud s \ud t,
\end{multline*}
where the first equality holds since $G_n(\un)(\cdot,t)\in L^{2}(\O)\cap W^{-1,p'}(\O)$
and $F_n^{k}(\un)(\cdot,t)\in L^{2}(\O)\cap W^{1,p}_{0}(\O)$ for a.e. $t\in (0,T)$.
Note that obtaining this $L^2$ integrability of $F_n^k(\un)$ is the
only reason for introducing the truncations; if $p\geq 2$ then the truncations are redundant.
As $k\to\infty$,
 $H_n^{k}(\un)\to H_n(\un)$ almost everywhere on $\O\times(0,T)$ and therefore
also in $\Leb{2}{2}$ by dominated convergence with \eqref{eq:nudom}. 
Thanks to G. Stampacchia's important result \cite{stam65}, we can write
$\nabla F_n^{k}(\un)=\trunc_{k}'(F_n(\un))\nabla F_n(u_n)
=\ch_{\{ |F_n(\un)|\leq k\}}\nabla F_n(\un)$, which converges in $\Leb[d]{p}{p}$ to
$\nabla F_n(\un)$ as $k\to\infty$. So $F_n^{k}(\un)\to F_n(\un)$ in 
$\LSobo{p}{1}{p}$. Let $k\to\infty$ on both sides of the above
inequality to obtain
\begin{multline*}
\int_{0}^{T-\tau}\int_{\O}\left(H_n(\un)(x,t+\tau) - H_n(\un)(x,t)\right)^{2}\ud x \ud t \\
\leq L_{F}L_{G}\int_{0}^{T-\tau}\int_{t}^{t+\tau}
\langle \partial_{t}G_n(\un)(\cdot,s), F_n(\un)(\cdot,t+\tau) - F_n(\un)(\cdot,t)
\rangle_{W^{-1,p'},W^{1,p}_{0}}\ud s \ud t\\
\leq L_{F}L_{G}\int_{0}^{T-\tau}\int_{t}^{t+\tau}
\norm{\partial_{t}G_n(\un)(\cdot,s)}_{W^{-1,p'}(\O)} \norm{F_n(\un)(\cdot,t+\tau) - F_n(\un)(\cdot,t)}_{W^{1,p}_{0}}
\ud s \ud t\\
\end{multline*}
Apply Young's inequality and interchange the order of
integration in $s$ and $t$ where appropriate to obtain
\begin{multline}\label{eq:tobecombined}
\norm{H_n(\un)(\cdot, \cdot + \tau) - H_n(\un)(\cdot, \cdot)}_{L^{2}(\O\times(0, T-\tau))}^{2}\\
\leq L_{F}L_{G}\tau\bigg(\frac{1}{p'}\norm{\partial_{t}G_n(\un)}_{\LSob{p'}{-1}{p'}}^{p'} 
+ \frac{2}{p}\norm{F_n(\un)}_{\LSobo{p}{1}{p}}^{p}\bigg)
\le \cter{cst:tr11}\tau
\end{multline}
where $\ctel{cst:tr11}$ does not depend on $n$ or $\tau$.
From the definition of $H_n$ we have $|H_n(\un)|\le L_F|G_n(\un)|$.
Hence, $(H_n(\un))_{n\in\NN}$ is bounded in $L^\infty(0,T;L^2(\O))$.
This enables us to estimate the time translates on $(0,\tau)$ and $(T-\tau,T)$ and,
combined with \eqref{eq:tobecombined} we deduce that \eqref{est:nutrans} holds for $r=2$.
The conclusion for a generic $r\in [2,\infty)$ follows by interpolation (H\"older's
inequality), using the bound of $(H_n(\un))_{n\in\NN}$ in $L^\infty(0,T;L^2(\O))$.
\end{proof}


\subsection{Two lemmas: convexity and monotonicity}\label{ssec:convmono}

Lemma \ref{lem:fatouconv2} is a general result on the
uniform weak lower semi-continuity of sequences of convex functions. 

\begin{lemma}\label{lem:fatouconv2}
Let $\Psi,\Psi_n:\RR\to [0,\infty]$ be convex lower semi-continuous functions
such that for every $n\in\NN$, $\Psi_n(0)=\Psi(0)=0$. Assume that for any $z\in\RR$ 
and any sequence $(z_n)_{n\in\NN}$ converging to $z$, $\Psi(z)\le \liminf_{n\to\infty} \Psi_n(z_n)$.
If $(v_{n})_{n\in\NN}\subset L^{2}(\O)$ converges weakly to $v$ in $L^{2}(\O)$ then
\begin{equation*}
\int_{\O}\Psi(v(x))\ud x \leq \liminf_{n\to\infty}\int_{\O}\Psi_n(\vn(x))\ud x.
\end{equation*}
\end{lemma}
\begin{proof}
For $x\in\O$, $r>0$ and $n\in\NN$, extend $\vn$ by zero outside $\O$ and write
\begin{equation*}
[\vn]_r(x) := \dashint_{B(x,r)} \vn(y)\ud y=\frac{1}{|B(x,r)|}\int_{B(x,r)}\vn(y)\ud y
\end{equation*}
for the mean value of $\vn$ over the closed ball of radius $r$ centred at $x$. Since
$\vn \weakto v$ in $L^{2}(\O)$ as $n\to\infty$, for every $x\in\O$,
\begin{align*}
[\vn]_r(x) 
\to \dashint_{B(x,r)}v(y)\ud y =: [v]_r(x).
\end{align*}
We have extended $v$ by $0$ outside $\O$. Hence
for every $x\in\O$, $\Psi([v]_r(x))\le\liminf_{n\to\infty}\Psi_n([\vn]_r(x))$.
We can apply Fatou's lemma and Jensen's inequality to obtain
\begin{align*}
\int_{\O}\Psi([v]_r(x))\ud x &\leq
\liminf_{n\to\infty}\int_{\O}\Psi_n([\vn]_r(x))\ud x
\leq \liminf_{n\to\infty}\int_{\O}\dashint_{B(x,r)}\Psi_n(\vn(y))\ud y \ud x.
\end{align*}
Use Fubini-Tonelli and the fact that $\Psi_n(v_n)=0$ outside $\O$
to write
\[
\int_{\O}\dashint_{B(x,r)}\Psi_n(\vn(y))\ud y \ud x 
= \int_{\O}\Psi_n(\vn(y))\ud y.
\]
Thus
\begin{equation}
\int_{\O}\Psi([v]_r(x))\ud x \leq \liminf_{n\to\infty}\int_{\O}\Psi_n(\vn(y))\ud y.\label{eq:almostfatou}
\end{equation}
Almost every $x\in\O$ is a Lebesgue point of $v$ and, for those, we have
$\lim_{r\to 0}[v]_r(x)=v(x)$.
Then from the lower semi-continuity of $\Psi$,
another application of Fatou's lemma and \eqref{eq:almostfatou}, 
we deduce that 
\begin{multline*}
\int_{\O}\Psi(v(x))\ud x \le \int_{\O}\liminf_{r\to 0}\Psi([v]_r(x))\ud x
 \leq \liminf_{r\to 0}\int_{\O}\Psi([v]_r(x))\ud x\\
 \leq \liminf_{r\to 0}\liminf_{n\to\infty}\int_{\O}\Psi_n(\vn(y))\ud y
 = \liminf_{n\to\infty}\int_{\O}\Psi_n(\vn(x))\ud x.\qedhere
\end{multline*} 
\end{proof}

We employ the next result to identify weak nonlinear limits in Section \ref{ssec:step2}. 

\begin{lemma}\label{lem:mintylike}
Let $V$ be a measurable subset of $\RR^d$. Take sequences 
$(\chin)_{n\in\NN}$, $(\psin)_{n\in\NN}\subset C(\RR)$ of nondecreasing functions
satisfying $\chin(0)=\psin(0)=0$ for every $n\in\NN$ and such that 
$\chin\to\chi$ and $\psin\to\psi$, both pointwise on $\RR$. 
Assume there is a sequence $(\vn)_{n\in\NN}$ of measurable functions on $V$ and two
functions $\chitilde$, $\psitilde\in L^{2}(V)$ such that
\begin{enumerate}[(i)]
\item $\chin(\vn)\weakto\chitilde$ and $\psin(\vn)\weakto\psitilde$, both weakly in $L^{2}(V)$;
\item there exists an almost-everywhere strictly positive function
$\varphi\in L^{\infty}(V)$ such that
\begin{equation*}
\lim_{n\to\infty}\int_{V}\varphi(z)\chin(\vn(z))\psin(\vn(z))\ud z
= \int_{V}\varphi(z)\chitilde(z)\psitilde(z)\ud z.
\end{equation*}
\end{enumerate}
Then for all measurable functions $v$ satisfying $(\chi+\psi)(v) = \chitilde+\psitilde$
almost everywhere in $V$,
\begin{equation*}
\chitilde = \chi(v) \quad\mbox{and}\quad \psitilde = \psi(v) \quad \mbox{almost everywhere in }V.
\end{equation*}
\end{lemma}
\begin{proof}
Observe that $\chi(v), \psi(v)\in L^{2}(V)$ since by hypothesis they have the same sign, 
so that $|\chi(v)| + |\psi(v)| = |(\chi+\psi)(v)|=|\chitilde + \psitilde|\in L^{2}(V)$. 
Let $\trunc_k(s)=\min(k,\max(-k,s))$ be the truncation at level $k$.
Since $\chin$ and $\psin$ are nondecreasing, for the function $\varphi$ in (ii), write
\begin{equation}\label{mintylike0}
\int_{V}\varphi(z)\big[\chin(\vn(z)) - \chin(\trunc_k(v(z)))\big]
\big[\psin(\vn(z)) - \psin(\trunc_k(v(z)))\big]\ud z \geq 0.
\end{equation}
By their monotonicity and sign properties, the functions $\chin$ and $\psin$ are bounded on
$[-k,k]$ by $\max(|\chi_n(\pm k)|,|\psi_n(\pm k)|)$, which is
uniformly bounded with respect to $n$ since $\chin$ and $\psin$ converge
pointwise. Hence, by the dominated convergence theorem, $\chin(\trunc_k(v))\to \chi(\trunc_k(v))$
and $\psin(\trunc_k(v))\to \psi(\trunc_k(v))$ in $L^2(V)$ as $n\to\infty$.
Using (i) and (ii), we can therefore pass to the limit $n\to\infty$
in \eqref{mintylike0} and we find
\begin{equation}
\int_{V}\varphi(z)\big[\chitilde(z) - \chi(\trunc_k(v(z)))\big]
\big[\psitilde(z) - \psi(\trunc_k(v(z)))\big]\ud z \geq 0.\label{eq:mintylike1}
\end{equation}
The monotonicity and sign properties of $\chi$ and $\psi$ ensure that
$|\chi(\trunc_k(v))|\le |\chi(v)|$ and $|\psi(\trunc_k(v))|\le |\psi(v)|$.
Since $\chi(v)$ and $\psi(v)$ belong to $L^2(V)$, we deduce
that as $k\to\infty$, $\chi(\trunc_k(v))\to \chi(v)$ and $\psi(\trunc_k(v))\to \psi(v)$,
both in $L^2(V)$. Passing to the limit in \eqref{eq:mintylike1}, we obtain
\begin{equation}
\int_{V}\varphi(z)[\chitilde(z) - \chi(v(z))][\psitilde(z) - \psi(v(z))]\ud z \geq 0.\label{eq:mintylike}
\end{equation}
The identity $\chi(v) + \psi(v) = \chitilde + \psitilde$ gives
$\chitilde(z) - \chi(v(z)) = -(\psitilde(z) - \psi(v(z)))$,
which after substitution into \eqref{eq:mintylike} yields
\begin{equation*}
-\int_{V}\varphi(z)\left[\chitilde(z) - \chi(v(z))\right]^{2}\ud z  = -\int_{V}\varphi(z)\left[\psitilde(z) - \psi(v(z))\right]^{2}\ud z \geq 0.
\end{equation*}
From the positivity of $\varphi$ we conclude that
$\chitilde(z) = \chi(v(z))$ and $\psitilde(z) = \psi(v(z))$ for
almost every $z\in V$.
\end{proof}

\subsection{Uniform-temporal, weak-spatial compactness} \label{ssec:unicomp}

The classical Aubin--Simon compactness theorem --- an amalgamation of the work
of J.-P. Aubin \cite{aub63} and J. Simon \cite{sim87} --- does ensure uniform temporal
compactness in Lebesgue spaces (for the norm topology), provided that a spatial compactness estimate
is available in such spaces. This usually requires control of the gradients in Lebesgue
spaces. Since we lack such estimates on the gradient of $\beta(u)$, we must forfeit (at least
initially) strong compactness in the spatial variable.
We first recall a basic definition.
\begin{definition}\label{def:weakunifconv}
A sequence of continuous functions $v_{n}:[0,T]\to L^{2}(\O)\weak$ converges in the space
$C([0, T]; L^{2}(\O)\weak)$ to a function $v: [0,T] \to L^{2}(\O)$ if for all $\varphi\in L^{2}(\O)$,
the sequence of functions $[0,T]\ni t \mapsto \langle v_{n}(t), \varphi\rangle_{L^{2}(\O)}$
converges uniformly on $[0,T]$ to $[0,T]\ni t \mapsto \langle v(t), \varphi\rangle_{L^{2}(\O)}$
as $n\to\infty$.

Note that $v$ is then necessarily an element of $C([0,T];L^2(\O)\weak)$.
\end{definition}

\begin{proposition}\label{prop:unifweakcomp}
Let $(v_{n})_{n\in\NN}$ be a sequence of real-valued measurable functions
on $\O\times(0,T)$. Suppose that there exists $q>1$ and $R>0$ such that for every 
$n\in\NN$,
\begin{equation} \label{eq:estvn}
\sup_{t\in [0,T]}\norm{v_{n}(\cdot,t)}_{L^2(\O)}\leq R, \quad \norm{\partial_{t}v_{n}}_{\LSob{q}{-1}{1}}\leq R.
\end{equation}
Then $(v_{n})_{n\in\NN}$ is relatively compact in $C([0,T]; L^{2}(\O)\weak)$;
that is, there is a subsequence of $(v_{n})_{n\in\NN}$ that converges in the sense
of Definition \ref{def:weakunifconv}.
\end{proposition}

\begin{remark}
The space $W^{-1,1}(\O)$ has been chosen by convenience, but it
could be replaced with the dual space of any Banach space in which
$C^\infty_c(\O)$ is dense.
\end{remark}
\begin{proof}
Denote by $E$ the ball of radius $R$ in $L^{2}(\O)$, endowed with the weak topology.
Take $(\varphi_{l})_{l\in\NN}\subset C^{\infty}_{c}(\O)$ a dense sequence in 
$L^{2}(\O)$ and equip $E$ with the metric
\begin{equation*}
d_E(v,w)=\sum_{l\in\NN} \frac{\min(1,|\langle v-w,\varphi_l\rangle_{L^2(\O)}|)}{2^l}.
\end{equation*}
The $L^2(\O)$ weak topology on $E$ is the topology induced by this metric. 
The set $E$ is metric compact and therefore complete. The first bound in \eqref{eq:estvn}
ensures that every $v_n$ takes values in $E$. It remains to estimate
$d_{E}(v_{n}(s),v_{n}(s'))$. To this end, 
\begin{align*}
\left| \langle v_{n}(s') - v_{n}(s), \varphi_{l}\rangle_{L^{2}(\O)}\right|
&= \left| \int_{\O}(\vn(x,s')-\vn(x,s))\varphi_{l}(x)\ud x\right|\\
&= \left| \langle \vn(\cdot,s') - \vn(\cdot,s), \varphi_{l}\rangle_{W^{-1,1},W^{1,\infty}_{0}}\right|\\ 
&= \left| \int_{s}^{s'}\langle \partial_{t}v_{n}(\cdot,t),\varphi_{l}\rangle_{W^{-1,1},W^{1,\infty}_{0}}\ud t\right|\\
&\leq \norm{\partial_{t}v_{n}}_{\LSob{q}{-1}{1}}\norm{\mathbf{1}}_{L^{q'}(s,s')}\norm{\varphi_{l}}_{W^{1,\infty}_{0}(\O)}\\
&\leq R|s-s'|^{1/q'}\norm{\varphi_{l}}_{W^{1,\infty}_{0}(\O)}.
\end{align*}
Then
\begin{equation*}
d_{E}(v_{n}(s),v_{n}(s'))\leq 
\sum_{l\in\NN}2^{-l}\min\left(1,R|s-s'|^{1/q'}\norm{\varphi_{l}}_{W^{1,\infty}_{0}(\O)}\right)
=: \omega(s,s').
\end{equation*}
Dominated convergence for series then implies that $\omega(s,s')\to 0$ as 
$|s-s'|\to 0$. Hence, $(\vn)_{n\in\NN}$ belongs to $C([0,T];E)$
and is equi-continuous in that space. Invoking the Arzel\`a--Ascoli theorem 
and the compactness of $E$ in $L^2(\O)\weak$ completes the proof. 
\end{proof}


\section{Proof of the main result} \label{sec:proof}

We prove Theorem \ref{th:main} in five steps. In Step 1 we obtain compactness of the 
sequences of interest, and in Step 2 we identify the limits of these sequences.
In Step 3 we pass to the limit in \eqref{eq:solution}. Step 4 improves the temporal convergence
of $(\nun(\un))_{n\in\NN}$ to establish \eqref{eq:mainconv}. We conclude by establishing
the strong convergence \eqref{eq:strgrad} in Step 5. 

\subsection{Step 1: compactness results}\label{ssec:step1}

Apply Proposition \ref{prop:unifweakcomp} using Estimates \eqref{est:energy} on 
$(\betan(\un))_{n\in\NN}$ and $(\partial_{t}\betan(\un))_{n\in\NN}$, and Lemma \ref{lem:h} 
with $\Hn=\betan$, $\vn=\uinin$
to deduce the existence of 
$\betatilde\in C([0,T];L^{2}(\O)\weak)$ satisfying 
$\betatilde(\cdot,0) = \beta(\uini)$ in $L^{2}(\O)$ and such that up to a subsequence,
\begin{equation}\label{conv:weakunifbeta}
\ol{\betan(\un)}\weakto \betatilde \quad \mbox{ in }C([0,T]; L^{2}(\O)\weak). 
\end{equation}
From \eqref{est:energy}, up to a subsequence,
\begin{equation}
\zetan(\un)\weakto \zetatilde \quad \mbox{weakly in $\LSobo{p}{1}{p}$} \label{eq:zetaweak}
\end{equation}
for some function $\zetatilde\in \LSobo{p}{1}{p}$.
Next we obtain strong compactness of the sequence
$(\nun(\un))_{n\in\NN}$ by demonstrating that the translates in space and time vanish.
Recalling \eqref{eq:nubetatrans} and using a classical translate
estimate in $W^{1,p}_0(\O)$, for $\xi\in\RR^d$ and $q<p^*$,
\begin{multline}\label{nun:spacetr}
\norm{\nun(\un)(\cdot + \xi, \cdot) - \nun(\un)}_{L^{p}(0,T;L^{q}(\O))} \leq
L_{\beta}\norm{\zetan(\un)(\cdot+\xi,\cdot) - \zetan(\un)}_{L^{p}(0,T;L^q(\O))} \\
\leq \cter{cst:sptr11}\norm{\nabla \zetan(\un)}_{L^{p}(\O\times (0,T))^d}|\xi|^\theta
\leq \cter{cst:sptr11}\cter{energy}|\xi|^\theta,
\end{multline}
where $\theta>0$ and $\ctel{cst:sptr11}$ do not depend on $\xi$ or $n$,
and $\nun(\un)$ and $\zetan(\un)$ are extended by zero on the complement of $\O$.
But $|\nun(\un)|\le L_\zeta |\betan(\un)|$ and $(\nun(\un))_{n\in\NN}$ is therefore
bounded in $L^\infty(0,T;L^2(\O))$. Interpolated with \eqref{nun:spacetr}, this shows that,
for all $r<+\infty$,
\begin{equation}\label{nun:spacetr2}
\norm{\nun(\un)(\cdot + \xi, \cdot) - \nun(\un)}_{L^r(0,T;L^{\min(2,q)}(\O))} \leq
\cter{cst:sptr111}|\xi|^{\theta_r}
\end{equation}
where $\theta_r>0$ and $\ctel{cst:sptr111}$ do not depend on $\xi$ or $n$.
By the energy estimates \eqref{est:energy}, Lemma \ref{lem:nutrans} applied with $F_n=\zetan$
and $G_n=\betan$ shows that the time translates of $\nun(\un)$ converge
uniformly to zero in $L^r(0,T;L^2(\O))$ for all $r<+\infty$. 
Combined with \eqref{nun:spacetr2} and the Kolmogorov--M.Riesz--Fr\'echet compactness theorem,
this establishes that, up to a subsequence,
\begin{equation}
\nun(\un) \to \nutilde \quad \mbox{in }L^r(0,T;L^{\min(2,q)}(\O))\mbox{ for all
$r<+\infty$ and all $q<p^*$}. \label{eq:nunstrong}
\end{equation}
From the uniform growth of the sequence $(a_{n})_{n\in\NN}$ and \eqref{est:energy},
we assert the existence of $\atilde\in L^{p'}(\O\times(0,T))^d$ such that, up to a
subsequence,
\begin{equation}
a_{n}(\cdot,\nun(\un),\nabla\zetan(\un))\weakto\atilde \quad \mbox{weakly in }L^{p'}(\O\times(0,T))^{d}.\label{eq:anweak}
\end{equation}

\subsection{Step 2: identifying nonlinear weak limits}\label{ssec:step2}

We show that there exists a measurable $u$ such that
$\betatilde=\beta(u)$, $\zetatilde=\zeta(u)$ and $\nutilde=\nu(u)$.
Three separate analyses are required, depending on the case in Hypothesis \eqref{hyp:pbeta}.

\subsubsection{Case {\rm (I)}: $p\ge 2$.}~

Define $\mu = \beta + \zeta$, $\mun = \betan + \zetan$ and $\mutilde = \betatilde +
\zetatilde$. Fix a measurable function $u$ such that $(\mu + \nu)(u) = \mutilde + \nutilde$.
Such a $u$ exists since the hypotheses on $\beta$ and 
$\zeta$ ensure that the range of $\mu + \nu$ is all of $\RR$ and therefore the 
domain of the right inverse $(\mu + \nu)^{r}$ of $(\mu +\nu)$ (defined analogously to \eqref{def:betar})
is $\RR$. One possible choice for $u$ is then
$u=(\mu+\nu)^{r}(\mutilde+\nutilde)$. We now demonstrate that for such a $u$, 
$\betatilde = \beta(u)$, $\zetatilde = \zeta(u)$ and $\nutilde = \nu(u)$.

Using $p\ge 2$, the convergences \eqref{eq:zetaweak} and \eqref{eq:nunstrong}
ensure that $\zetan(\un)\weakto \zetatilde$ weakly in $L^2(\O\times (0,T))$, and that
$\nun(\un) \to \widetilde{\nu}$ strongly in $L^2(\O\times(0,T))$.
We deduce that $\mun(\un)=\betan(\un)+\zetan(\un)\weakto\betatilde+\zetatilde=\mutilde$ weakly
in $L^{2}(\O\times(0,T))$ and that 
\[
\int_{\O\times(0,T)}\mun(\un)(x,t)\nun(\un)(x,t)\ud x\ud t\to
\int_{\O\times(0,T)}\mutilde(x,t)\nutilde(x,t)\ud x\ud t.
\]
We can thus apply Lemma 
\ref{lem:mintylike} with $\varphi\equiv 1$, $\vn=\un$, $v=u$, $\chin=\mun$ and $\psin=\nun$ 
to deduce that $\nutilde=\nu(u)$ and $\mutilde=\mu(u)$ almost everywhere on $\O\times(0,T)$,
the latter of which states that
$(\beta+\zeta)(u) = \betatilde + \zetatilde$. 

Since $p\ge 2$, Estimates \eqref{est:energy} ensure that $(\zetan(\un))_{n\in\NN}$ and
$(\betan(\un))_{n\in\NN}$ satisfy the hypotheses of Lemma \ref{lem:compcomp}, and so
$\betan(\un)\zetan(\un)\weakto\betatilde\,\zetatilde$ in $(C(\overline{\O}\times[0,T]))'$.
Now as $(\beta+\zeta)(u) = \betatilde + \zetatilde$, we apply Lemma \ref{lem:mintylike}
again with $\varphi\equiv 1$, $\vn=\un$, $v=u$, $\chin=\betan$ and $\psin=\zetan$ to conclude that
$\betatilde = \beta(u)$ and $\zetatilde = \zeta(u)$ almost everywhere on $\O\times(0,T)$.

\subsubsection{Case {\rm (II)}: $\frac{2d}{d+2}<p<2$ and $|\betan(s)|\ge M_3|s|-M_4$.}~

Since $(\betan(\un))_{n\in\NN}$ is bounded in $L^\infty(0,T;L^2(\O))$, the assumption
on $\betan$ shows that $(\un)_{n\in\NN}$ is bounded in the same space. By the uniform
Lipschitz continuity of $\zetan$, we infer that $(\zetan(\un))_{n\in\NN}$ is also bounded
in $L^\infty(0,T;L^2(\O))$. Hence, as in the previous case the convergence \eqref{eq:zetaweak}
also holds weakly in $L^2(\O\times (0,T))$.
Since $p^*>2$, \eqref{eq:nunstrong} gives the strong convergence of
$\nun(\un)$ in $L^2(\O\times(0,T))$.

We proceed as in the previous case to see that with $u=(\mu+\nu)^r(\mutilde
+\nutilde)$, $\nu(u)=\nutilde$ and $\beta(u)+\zeta(u)=\betatilde+\zetatilde$.
Now apply Lemma \ref{lem:compcomp} to $(\zetan(\un))_{n\in\NN}$
and $(\betan(\un))_{n\in\NN}$. As in Case (I), this gives 
$\betatilde = \beta(u)$ and $\zetatilde = \zeta(u)$.

\subsubsection{Case {\rm (III)}: $1<p\le \frac{2d}{d+2}$, $|\betan(s)|\ge M_3|s|-M_4$ and
$\beta$ is strictly increasing}~

As in Case (II), the coercivity assumption on $\betan$ ensures that $(\zetan(\un))_{n\in\NN}$
converges weakly in $L^2(\O\times (0,T))$. However, we can no longer ensure
the strong convergence of $\nun(\un)$ in $L^2$.
We must therefore truncate $\zetan$ first. Let $\zetan^k=\trunc_k(\zetan)$,
where $\trunc_k(s)=\min(k,\max(-k,s))$ is the truncation at level $k$.
Up to a subsequence, for some $\zetatilde^k\in L^2(\O\times(0,T))$, $\zetan^k(\un)\weakto \zetatilde^k$
weakly in $L^2(\O\times (0,T))$. Set
\begin{equation*}
\nun^k(s)=\int_0^s \betan'(q)(\zetan^k)'(q)\ud q.
\end{equation*}
Note that $(\nabla\zetan^k(\un))_{n\in\NN}=(\ch_{\{|\zetan(\un)|\le k\}}\nabla\zetan(\un))_{n\in\NN}$ is bounded in $L^p(\O\times (0,T))^d$. Hence, following the reasoning in
\eqref{nun:spacetr} and using an interpolation in space between $p$ and $\infty$ (we have
$|\zetan^k|\le k$),
\begin{align*}
\norm{\nun^k(\un)(\cdot + \xi, \cdot) - \nun^k(\un)}_{L^{p}(0,T;L^2(\O))} &\leq
L_{\beta}\norm{\zetan^k(\un)(\cdot+\xi,\cdot) - \zetan^k(\un)}_{L^{p}(0,T;L^2(\O))} \\
&\leq L_{\beta}(2k)^{1-\frac{p}{2}}\norm{\zetan^k(\un)(\cdot+\xi,\cdot) - \zetan^k(\un)}_{L^{p}(0,T;L^p(\O))}^{\frac{p}{2}} \\
&\leq \cter{cst:sptr12}\norm{\nabla \zetan^k(\un)}_{L^{p}(\O\times (0,T))^d}^{\frac{p}{2}}|\xi|^{\frac{p}{2}}
\leq \cter{cst:sptr13}|\xi|^{\frac{p}{2}},
\end{align*}
where $\ctel{cst:sptr12}$ and $\ctel{cst:sptr13}$ depend on $k$ but not on $n$ or $\xi$.
Use the bound on $(\nun^k(\un))_{n\in\NN}$ in $L^\infty(0,T;L^2(\O))$ to infer
that the space translates of these functions vanish uniformly with respect to $n$
in $L^r(0,T;L^2(\O))$ for all $r<+\infty$.
Lemma \ref{lem:nutrans} applied to $F_n=\zetan^k$ and $G_n=\betan$ shows that 
the time translates of $\nun^k(\un)$ vanish uniformly with respect to $n$
in $L^r(0,T;L^2(\O))$ for all $r<+\infty$. Hence, $(\nun^k(\un))_{n\in\NN}$ 
strongly converges, up to a subsequence, to some $\nutilde^k$ in $L^2(\O\times (0,T))$.

We can then work as in the previous cases with $\betan$, $\zetan^k$ and $\nun^k$.
We define $\zeta^k=\trunc_k(\zeta)$ and $\nu^k(s)=\int_0^s \beta'(q)(\zeta^k)'(q)\ud q$, and we
let $\mu^k=\beta+\zeta^k$.
By coercivity of $\beta$, the mapping $\mu^k+\nu^k$ is onto and we can
define $u^k=(\mu^k+\nu^k)^r(\mutilde^k+\nutilde^k)$, where $\mutilde^k=\betatilde
+\zetatilde^k$ is the weak limit in $L^2(\O\times (0,T))$ of $\betan+\zetan^k$.
By strong convergence in $L^2(\O\times(0,T))$ of $(\nun^k(\un))_{n\in\NN}$ we can apply Lemma 
\ref{lem:mintylike} to see that $\nutilde^k=\nu^k(u^k)$ and
\begin{equation}\label{rel:k}
\betatilde+\zetatilde^k=\beta(u^k)+\zeta^k(u^k).
\end{equation}

We now apply Lemma \ref{lem:compcomp}
to $(\zetan^k(\un))_{n\in\NN}$ and $(\betan(\un))_{n\in\NN}$.
Indeed, $(\zetan^k(\un))_{n\in\NN}$ is bounded in $L^\infty(\O\times(0,T))$. We therefore obtain $\betan(\un)\zetan^k(\un)
\weakto \betatilde\,\zetatilde^k$ weakly in $C(\overline{\O}\times [0,T])'$.
Use Lemma \ref{lem:mintylike} and \eqref{rel:k}
to deduce that $\zetatilde^k=\zeta^k(u^k)$ and $\betatilde=\beta(u^k)$.

Since $\beta$ does not have any plateaux and $\betatilde$ does not depend on $k$,
the latter relation shows that $u^k$ does not depend on $k$.
Write $u=u^k$. Then $\betatilde=\beta(u)$, $\zetatilde^k=\trunc_k(\zeta(u))$
and $\nutilde^k=\nu^k(u)$. If we can show that $\zetatilde^k\to \zetatilde$
and $\nutilde^k\to \nutilde$ in $\mathcal D'(\O\times (0,T))$ as $k\to\infty$, then we can
pass to the limit in the previous equalities to get
$\zetatilde=\zeta(u)$ and $\nutilde=\nu(u)$, as required.

Begin with the convergence of $\zetatilde^k$. This function is the weak
limit in $L^2(\O\times (0,T))$ of $(\zetan^k(\un))_{n\in\NN}$. By Tchebycheff's inequality,
uniform Lipschitz continuity of $\zetan^k$ and the bound of
$(\un)_{n\in\NN}$ in $L^\infty(0,T;L^2(\O))$ (from the coercivity of $\betan$),
\begin{equation}\label{est:meas}
\mbox{meas}(\{|\zetan^k(\un)|\ge k\})\le \frac{\cter{cst:tcheb}}{k}
\end{equation}
with $\ctel{cst:tcheb}$ not depending on $k$ or $n$.
Let $\varphi\in C^\infty_c(\O\times (0,T))$. Then
\begin{multline}
\left|\int_{\O\times (0,T)}[\zetatilde^k-\zetatilde(x,t)]\varphi(x,t)\ud x\ud t\right|\\
\leq
\left|\int_{\O\times (0,T)}[\zetatilde^k-\zetan^k(\un)](x,t)\varphi(x,t)\ud x\ud t\right|
+\left|\int_{\O\times (0,T)}[\zetan^k(\un)-\zetan(\un)](x,t)\varphi(x,t)\ud x\ud t\right|\\
+\left|\int_{\O\times (0,T)}[\zetan(\un)-\zetatilde](x,t)\varphi(x,t)\ud x\ud t\right|.
\label{limsuptotake}\end{multline}
By \eqref{eq:zetaweak}, the last term tends to $0$ as $n\to\infty$. The first term
also vanishes as $n\to\infty$. Estimate the second term using
$|\zetan^k(\un)-\zetan(\un)|\le \ch_{\{|\zetan(\un)|\ge k\}}|\zetan(\un)|$,
H\"older's inequality, the energy estimate \eqref{est:energy}
and the inequality \eqref{est:meas}. Taking the limit superior as $n\to\infty$ of
\eqref{limsuptotake} yields
\begin{equation*}
\left|\int_{\O\times (0,T)}[\zetatilde^k-\zetatilde(x,t)]\varphi(x,t)\ud x\ud t\right|\\
\leq \cter{energy}||\varphi||_{L^{\infty}(\O\times (0,T))}
\left(\frac{\cter{cst:tcheb}}{k}\right)^{1/p'}.
\end{equation*}
Letting $k\to\infty$ concludes the proof that $\zetatilde^k\to \zetatilde$
in the sense of distributions.

The proof that $\nutilde^k$ converges as $k\to \infty$ to $\nutilde$ in the sense of 
distributions is similar. The functions $\nutilde^k$ and $\nutilde$
are the weak limits in $L^2(\O\times (0,T))$ of $(\nun^k(\un))_{n\in\NN}$
and $(\nun(\un))_{n\in\NN}$ (note that since the latter sequence is bounded
in $L^\infty(0,T;L^2(\O))$, the convergence \eqref{eq:nunstrong} also holds
weakly in this space). Moreover 
\begin{equation*}
\nun^k(\un)-\nun(\un)=\int_0^{\un} \betan'(q)\left(\trunc_k(\zetan)-\zetan\right)'(q)\ud q
=0\quad\mbox{ if $|\un|\le k$}.
\end{equation*}
We can therefore reproduce the same reasoning as for the convergence of $(\zetatilde^k)_{k\to\infty}$
to see that $\nutilde^k\to \nutilde$ in the sense of distributions as $k\to\infty$.

\begin{remark}\label{rem:compcomp}If $\betan=\Id$ (resp. $\zetan=\Id$), then $\nun=\zetan$
(resp. $\nun=\betan$) and the strong convergence of $\nun(\un)$ enables us to pass to the limit
in $\int_{\O\times (0,T)}\betan(\un)\zetan(\un)$ (or the truncated version if $p$ is small).
We only need the compensated compactness lemma to identify this limit in the case
of two genuine degeneracies, that is $\betan\not=\Id$ and $\zetan\not=\Id$.

\end{remark}


\subsection{Step 3: the function $u$ is a solution to \eqref{eq:model}}\label{ssec:step3}

We know that $\zeta(u)=\zetatilde\in L^p(0,T;W^{1,p}_0(\O))$, $\beta(u)=\betatilde
\in C([0,T];L^2(\O)\weak)$ (with an abuse of notation), $\ol{\beta(u)}(\cdot,0)=
\betatilde(\cdot,0)=\beta(\uini)$. Since $(\partial_t\betan(\un))_{n\in\NN}$
is bounded in $L^{p'}(0,T;W^{-1,p'}(\O))$, we infer that $\partial_t \betan(\un)
\weakto \partial_t \beta(u)$ weakly in this space.
Lemma \ref{lem:convBn} shows that $\Psi=B$ and $\Psi_n=B_n$
satisfy the assumptions of Lemma \ref{lem:fatouconv2}. 
Let $T_0 \in [0,T]$. By \eqref{conv:weakunifbeta}, $\ol{\betan(\un)}(\cdot,T_0)
\weakto \ol{\beta(u)}(\cdot,T_0)$ weakly in $L^2(\O)$. Hence by Lemma \ref{lem:fatouconv2},
\begin{equation}\label{eq:minorliminf}
\int_{\O}B(\ol{\beta(u)}(x,T_0))\ud x\leq \liminf_{n\to\infty}\int_{\O}\Bn(\ol{\betan(\un)}(x,T_{0}))\ud x.
\end{equation}
Combined with \eqref{est:energy}, this shows that $B(\beta(u))\in L^\infty(0,T;L^1(\O))$.

Passing to the limit as $n\to\infty$ in \eqref{eq:solution} is then possible
thanks to the convergence properties of $\partial_t \betan(\un)$ and $a_n(\cdot,\nun(\un),
\nabla \zetan(\un))$.
We obtain
\begin{multline}
\int_{0}^{T}\langle \partial_{t}\beta(u)(\cdot,t),v(\cdot,t)\rangle_{W^{-1,p'},W^{1,p}_{0}}\ud t 
+ \int_{0}^{T}\int_{\O}\atilde(x,t)\cdot\nabla v(x,t)\ud x \ud t \\
= \int_{0}^{T}\langle f(\cdot,t),v(\cdot,t)\rangle_{W^{-1,p'},W^{1,p}_{0}}\ud t \quad \forall v\in \LSobo{p}{1}{p}. \label{eq:step2sol2}
\end{multline}

To complete Step 3, it remains to demonstrate that
\begin{equation}
\atilde(x,t) = a(x,\nu(u),\nabla \zeta(u))(x,t) \quad \mbox{for a.e. }(x,t)\in\O\times(0,T). \label{eq:Aequalsa}
\end{equation}
Let $T_{0}\in[0,T]$ and consider the identity
\eqref{eq:energy} with data $(\betan,\zetan,\nun,\an,\fn,\uinin)$.
Take the limit superior and
use \eqref{eq:zetaweak} (recall that $\zetatilde=\zeta(u)$) to obtain
\begin{multline}\label{eq:alimsup}
\limsup_{n\to\infty}\int_{0}^{T_{0}}\int_{\O}\an(x,\nun(\un),\nabla\zetan(\un))\cdot\nabla\zetan(\un)\ud x \ud t \\
\leq \limsup_{n\to\infty}\int_{\O}\Bn(\betan(\uinin(x)))\ud x 
+ \int_{0}^{T_{0}}\langle f(\cdot,t),\zeta(u)(\cdot,t)\rangle_{W^{-1,p'},W^{1,p}_{0}}\ud t \\
- \liminf_{n\to\infty}\int_{\O}\Bn(\ol{\betan(\un)}(x,T_{0}))\ud x.
\end{multline}
Part (\ref{convBn:it2}) of Lemma \ref{lem:convBn} and \eqref{eq:growthB} show that
$B_n\circ \betan$ converges uniformly and has uniform quadratic growth. By
applying Lemma \ref{lem:h}, the convergence $\uinin\to \uini$ in $L^2(\O)$
shows that $(\Bn\circ\betan)(\uinin)\to (B\circ\beta)(\uini)$ in $L^{1}(\O)$.
Together with the inequality \eqref{eq:minorliminf}, this gives 
\begin{multline}\label{atilde:limsup}
\limsup_{n\to\infty}\int_{0}^{T_{0}}\int_{\O}\an(x,\nun(\un),\nabla\zetan(\un))\cdot\nabla\zetan(\un)\ud x \ud t \\
\leq \int_{\O}B(\beta(\uini(x)))\ud x 
+ \int_{0}^{T_{0}}\langle f(\cdot,t),\zeta(u)(\cdot,t)\rangle_{W^{-1,p'},W^{1,p}_{0}}\ud t \\
- \int_{\O}B(\ol{\beta(u)}(x,T_{0}))\ud x = \int_{0}^{T_{0}}\int_{\O}\atilde(x,t)\cdot\nabla\zeta(u)\ud x \ud t,
\end{multline}
using the identities \eqref{eq:integparts} (with $v=u$) and \eqref{eq:step2sol2} (with
$v=\zeta(u)$).

We now employ the classical Minty--Browder argument. For $G\in \Leb[d]{p}{p}$, the monotonicity
of $(\an)_{n\in\NN}$ gives
\begin{equation}
\int_{0}^{T_{0}}\int_{\O}[\an(x,\nun(\un),\nabla\zetan(\un))-\an(x,\nun(\un),G)]\cdot[\nabla\zetan(\un)-G]\ud x\ud t\geq 0.
\label{eq:anmonotone}
\end{equation}
Together with the strong convergence in $L^{1}(\O\times(0,T))$ of $\nun(\un)$ to $\nu(u)$,
the assumptions on the sequence $(\an)_{n\in\NN}$ ensure that
$\an(\cdot,\nun(\un),G)$ converges in $L^{p'}(\O\times(0,T))^{d}$ to $a(\cdot,\nu(u),G)$.
Using this, \eqref{atilde:limsup} and the weak convergence \eqref{eq:anweak}, we pass to the limit superior on the 
expanded form of \eqref{eq:anmonotone} with $T_{0}=T$ to see that
\begin{equation*}
\int_{0}^{T_{0}}\int_{\O}[\atilde(x,t) - a(x,\nu(u(x,t)),G(x,t))]
\cdot[\nabla\zeta(u)(x,t)-G(x,t)]\ud x\ud t \geq0.
\end{equation*}
Following G.J. Minty \cite{minty}, take $G=\nabla\zeta(u) \pm r\varphi$ for 
$\varphi\in\Leb[d]{p}{p}$, divide by $r>0$ and let $r\to 0$ to obtain \eqref{eq:Aequalsa}.


\subsection{Step 4: uniform temporal convergence of $\nun(\un)$ to $\nu(u)$}
\label{ssec:step4}
Take $T_\infty\in[0,T]$ and $(T_{n})_{n\in\NN}\subset[0,T]$ a sequence converging to
$T_\infty$. Thanks to Lemma \ref{lem:equiv-unifconv}, the convergence of
$(\ol{\nun(\un)})_{n\in\NN}$ in $C([0,T];L^2(\O))$ follows if we can demonstrate that
\begin{equation}\label{eq:nununiform}
\lim_{n\to\infty}\norm{\ol{\nun(\un)}(\cdot,T_{n}) - \ol{\nu(u)}(\cdot,T_\infty)}_{L^{2}(\O)}=0.
\end{equation}
Note the use of the continuous representatives $[0,T]\to L^2(\O)$
of $\nun(\un)$ and $\nu(u)$ (whose existence is ensured by Lemma \ref{lem:integparts}).
Without loss of generality, we can assume that $T_n$ is such that
\begin{equation}\label{eq:ae}
\betan(\un(\cdot, \Tn))=\ol{\betan(\un)}(\cdot, \Tn)\mbox{ and }
\nun(\un(\cdot, \Tn)) = \ol{\nun(\un)}(\cdot, \Tn)\mbox{ a.e. on $\O$}.
\end{equation}
Indeed, by definition of the continuous representatives,
there is $T_n'\in (T_n-1/n,T_n+1/n)\cap [0,T]$ such that \eqref{eq:ae}
holds at $T_n'$ and such that 
\[
\norm{ \ol{\nun(\un)}(\cdot, \Tn) 
- \ol{\nun(\un)}(\cdot, \Tn') }_{L^{2}(\O)} \leq \frac{1}{n},
\]
using
$\ol{\nun(\un)}\in C([0,T];L^2(\O))$. Proving \eqref{eq:nununiform} with $T_n'$
instead of $T_n$ establishes it for $T_n$ also.

To estimate the quantity in \eqref{eq:nununiform}, which involves a variation
of $\nu_n$ and $\un$ with respect to $n$, our strategy is to freeze one of these
variations using the triangle inequality with $\nun(u)(\cdot,T_\infty)$ as an intermediate
point. But $\nun(u)$ may not be continuous in time, so its value at $T_\infty$ is not well-defined. 
Instead we use $\nun(u)(\cdot,s)$ and average over a small interval around $T_\infty$. 
To this end, let $\eps>0$ and define
$I_{\eps}:= [T_\infty-\eps, T_\infty+\eps]\cap[0,T]$. 
Using \eqref{eq:ae} and \eqref{eq:Bunifconvex} with $\nun$, $B_n$ and $\betan$, write
\begin{align}
\Big\lVert&\ol{\nun(\un)}(\cdot,\Tn) - \ol{\nu(u)}(\cdot,T_\infty)\Big\rVert_{L^{2}(\O)}^{2}\label{eq:Iaverage}\\
\leq{}& 2\,\dashint_{I_{\eps}} \norm{\nun(\un(\cdot,\Tn)) - \nun(u(\cdot,s))}_{L^{2}(\O)}^{2}\ud s
+ 2\,\dashint_{I_{\eps}}\norm{\nun(u(\cdot,s)) - \ol{\nu(u)}(\cdot,T_\infty)}_{L^{2}(\O)}^{2}\ud s\nonumber\\
\leq{}& 8L_{\beta}L_{\zeta} \bigg(\int_{\O}\Bn(\betan(\un(x,\Tn)))\ud x 
+ \dashint_{I_{\eps}}\int_{\O}\Bn(\betan(u(x,s)))\ud x \ud s\nonumber\\
&-2\,\dashint_{I_{\eps}}\int_{\O}\Bn\left(\frac{\betan(\un(x,\Tn)) + \betan(u(x,s))}{2}\right)\ud x \ud s\bigg)\nonumber\\
&+ 2\,\dashint_{I_{\eps}}\norm{\nun(u(\cdot,s)) - \ol{\nu(u)}(\cdot,T_\infty)}_{L^{2}(\O)}^{2}\ud s\nonumber\\
=:{}& 8L_{\beta}L_{\zeta} [\term_{1}(n) + \term_{2}(n,\eps) - 2\term_{3}(n,\eps)] + 2\term_{4}(n,\eps).\nonumber
\end{align}
To determine the convergence of $\term_{1}$, expand \eqref{eq:anmonotone} with $T_{0} = T_n$,
$G=\nabla\zeta(u)$ and take the limit inferior of the resulting expression. Noting the 
identity \eqref{eq:Aequalsa}, we obtain
\begin{multline}
\liminf_{n\to\infty}\int_{0}^{T_n}\int_{\O}\an(x,\nun(\un),\nabla\zetan(\un))\cdot\nabla\zetan(\un)\ud x \ud t\\
\geq \int_{0}^{T_\infty}\int_{\O} a(x,\nu(u),\nabla\zeta(u))\cdot\nabla\zeta(u)\ud x \ud t. \label{eq:liminfan}
\end{multline}
Now in \eqref{eq:energy}, replace 
$(\beta,\zeta,\nu,a,f,\uini, T_{0})$ with $(\betan,\zetan,\nun,\an,\fn,\uinin, T_{n})$
and using \eqref{eq:ae}, \eqref{eq:liminfan} and the fact that $u$ satisfies the energy
equality \eqref{eq:energy}, take the limit superior as $n\to\infty$
to deduce that
\begin{equation} \label{eq:T1}
\limsup_{n\to\infty} \term_{1}(n)
\leq \int_{\O}B(\ol{\beta(u)}(x,T_\infty))\ud x<+\infty.
\end{equation}
To handle $\term_{2}$, recall that $B_n\circ\betan$ converges locally uniformly on $\RR$ to
to $B\circ\beta$ (Lemma \ref{lem:convBn}). By Hypothesis \eqref{hyp:pbeta},
$u\in L^2(\O\times (0,T))$. Hence by the dominated convergence theorem,
the quadratic growth \eqref{eq:growthB} of $B_n$ ensures that
$\Bn(\betan(u))\to B(\beta(u))=B(\ol{\beta(u)})$ in $L^{1}(\O\times(0,T))$ and so 
\begin{equation} \label{eq:T2}
\lim_{n\to\infty} \term_{2}(n,\eps) = \dashint_{I_{\eps}}\int_{\O}B(\ol{\beta(u)}(x,s))\ud x\ud s.
\end{equation}
Now the convexity of $\Bn$ enables the application of Jensen's inequality to $\term_{3}$, yielding
\begin{multline*}
\term_{3}(n,\eps)=
\dashint_{I_{\eps}}\int_{\O}\Bn\left(\frac{\betan(\un(x,\Tn)) + \betan(u(x,s))}{2}\right)\ud x \ud s\\
\geq \int_{\O}\Bn\left( \frac{\betan(\un(x,\Tn)) + \dashint_{I_{\eps}}\betan(u(x,s))\ud s}{2}\right)\ud x.
\end{multline*}
The convergence in $C([0,T]; L^{2}(\O)\weak)$ of $\ol{\betan(\un)}$ to $\ol{\beta(u)}$ and the continuity
of the latter imply by Lemma \ref{lem:equiv-unifconv} that
$\betan(\un(\cdot,\Tn))=\ol{\betan(\un)}(\cdot, \Tn)\weakto \ol{\beta(u)}(\cdot, T_\infty)$ weakly in $L^{2}(\O)$.
Since $u\in L^2(\O\times (0,T))$ the assumptions on $\betan$
give $\betan(u)\to\beta(u)=\ol{\beta(u)}$ in $\Leb{2}{2}$, and so
\begin{equation*}
\dashint_{I_{\eps}}\betan(u(\cdot,s))\ud s \to \dashint_{I_{\eps}}\ol{\beta(u)}(\cdot,s)\ud s \quad \mbox{in $L^{2}(\O)$.}
\end{equation*}
Thus 
$\frac{1}{2}( \betan(\un(,\cdot, \Tn)) + \dashint_{I_{\eps}}\betan(u(\cdot,s))\ud s)
\weakto \frac{1}{2}(\ol{\beta(u)}(\cdot,T_\infty) + \dashint_{I_{\eps}}\ol{\beta(u)}(\cdot,s)\ud s)$
weakly in $L^{2}(\O)$ and Lemma \ref{lem:fatouconv2} gives
\begin{equation} \label{eq:T3}
\int_{\O}B\left( \frac{\ol{\beta(u)}(x,T_\infty) + \dashint_{I_{\eps}}\ol{\beta(u)}(x,s)\ud s}{2}\right)\ud x
\leq \liminf_{n\to\infty}\term_{3}(n, \eps).
\end{equation}
Since $u\in L^2(\O\times(0,T))$, $\nun(u)\to \nu(u)$ in $L^2(\O\times(0,T))$
and so
\begin{multline} \label{eq:T4}
\term_{4}(n, \eps)
= \frac{1}{|I_{\eps}|}\norm{\nun(u) - \ol{\nu(u)}(\cdot,T_\infty)}_{L^{2}(\O\times I_{\eps})}^{2}
\to \frac{1}{|I_{\eps}|}\norm{\nu(u) - \ol{\nu(u)}(\cdot,T_\infty)}_{L^{2}(\O\times I_{\eps})}^{2}\\
= \dashint_{I_{\eps}}\norm{\ol{\nu(u)}(\cdot,s) - \ol{\nu(u)}(\cdot,T_\infty)}_{L^{2}(\O)}^{2}\ud s.
\end{multline}
Thanks to \eqref{eq:T1}, \eqref{eq:T2} and \eqref{eq:T3}, we may split the limit superior as 
$n\to\infty$ of the right-hand side of \eqref{eq:Iaverage}, using \eqref{eq:T4} 
for the remaining term to obtain
\begin{multline}\label{eq:Iaverage2}
\limsup_{n\to\infty}\norm{\ol{\nun(\un)}(\cdot,T_{n}) - \ol{\nu(u)}(\cdot,T_\infty)}_{L^{2}(\O)}^{2}\\
\leq 8L_{\beta}L_{\zeta}\bigg( \int_{\O}B(\ol{\beta(u)}(x,T_\infty)\ud x 
+ \dashint_{I_{\eps}}\int_{\O}B(\ol{\beta(u)}(x,s))\ud x \ud s\\ 
- 2\int_{\O}B\left( \frac{\ol{\beta(u)}(x,T_\infty) + \dashint_{I_{\eps}}
\ol{\beta(u)}(x,s)\ud s}{2}\right)\ud x\bigg)
+ 2\,\dashint_{I_{\eps}}\norm{\ol{\nu(u)}(\cdot,s) - \ol{\nu(u)}(\cdot,T_\infty)}_{L^{2}(\O)}^{2}\ud s.
\end{multline}
To complete the proof it remains to take the superior limit as $\eps\to 0$. By the 
continuity of the mapping $[0,T]\ni s \mapsto \int_{\O}B(\ol{\beta(u)}(x,s))\ud x$,
\begin{equation*}
\lim_{\eps\to 0}\,\dashint_{I_{\eps}}\int_{\O}B(\ol{\beta(u)}(x,s))\ud x \ud s 
= \int_{\O}B(\ol{\beta(u)}(x,T_\infty))\ud x.
\end{equation*}
Using the continuity of $\ol{\nu(u)}: [0,T]\to L^{2}(\O)$,
\begin{equation*}
\lim_{\eps\to 0}\,\dashint_{I_{\eps}}\norm{\ol{\nu(u)}(\cdot,s)-\ol{\nu(u)}(\cdot,T_\infty)}_{L^{2}(\O)}^{2}\ud s = 0.
\end{equation*}
Since $B$ is convex lower semi-continuous and $\frac{1}{2}(\ol{\beta(u)}(\cdot,T_\infty) +
\dashint_{I_{\eps}}\ol{\beta(u)}(\cdot,s)\ud s)
\weakto \ol{\beta(u)}(\cdot,T_\infty)$ weakly in $L^{2}(\O)$ as $\eps\to 0$
(using the continuity of $\ol{\beta(u)}:[0,T]\to L^2(\O)\weak$), we apply 
Lemma \ref{lem:fatouconv2} to deduce that
\begin{equation*}
\int_{\O}B(\ol{\beta(u)}(x,T_\infty))\ud x
\leq \liminf_{\eps\to 0}\int_{\O}B\left( \frac{\ol{\beta(u)}(x,T_\infty) + \dashint_{I_{\eps}}
\ol{\beta(u)}(x,s)\ud s}{2}\right)\ud x.
\end{equation*}
Taking the limit supremum as $\eps\to 0$ of \eqref{eq:Iaverage2} yields \eqref{eq:nununiform},
hence the result.

\begin{remark} Since $\ol{\betan(u_n)}(\cdot,T_n)\weakto \ol{\beta(u)}(\cdot,T_\infty)$ weakly in $L^2(\Omega)$
whenever $T_n\to T_\infty$ (see Lemma \ref{lem:equiv-unifconv}),
\eqref{eq:minorliminf} still holds with $T_0$ in the left-hand side replaced
with $T_\infty$ and $T_0$ in the right-hand side replaced with $T_n$. Thus with \eqref{eq:T1} we see that
\[
\int_\O B_n(\ol{\betan(u_n)}(x,T_n))\ud x\to \int_\O B(\ol{\beta(u)}(x,T_\infty))\ud x\quad\mbox{as $n\to\infty$.}
\]
Lemma \ref{lem:equiv-unifconv} and Part (i) in Lemma \ref{lem:integparts} then
show that $\int_\O B_n(\ol{\betan(u_n)}(x,\cdot))\ud x$ converges uniformly to
$\int_\O B(\ol{\beta(u)}(x,\cdot))\ud x$ on $[0,T]$.
\end{remark}


\subsection{Step 5: convergence of $\zetan(\un)$ to $\zeta(u)$ in $\LSobo{p}{1}{p}$.}
\label{ssec:step5}
We follow the ideas of J. Leray and J.-L. Lions \cite{ll65}.
Use \eqref{atilde:limsup} with $T_0=T$ and \eqref{eq:Aequalsa}:
\[
\limsup_{n\to\infty}\int_{0}^{T}\int_{\O}\an(x,\nun(\un),\nabla\zetan(\un))\cdot\nabla\zetan(\un)\ud x \ud t
\leq
\int_{0}^{T}\int_{\O}a(x,\nu(u),\nabla\zeta(u))\cdot\nabla\zeta(u)\ud x \ud t.
\]
Together with $T_n = T_\infty = T$ in \eqref{eq:liminfan}, we see that
\begin{multline}
\lim_{n\to\infty}\int_{0}^{T}\int_{\O}\an(x,\nun(\un),\nabla\zetan(\un))\cdot\nabla\zetan(\un)\ud x\ud t\\
= \int_{0}^{T}\int_{\O}a(x,\nu(u),\nabla\zeta(u))\cdot\nabla\zeta(u)\ud x \ud t.\label{eq:limintan}
\end{multline}
Now define
\begin{equation*}
F_{n}:=[\an(x,\nun(\un),\nabla\zetan(\un))-\an(x,\nun(\un),\nabla\zeta(u))]\cdot[\nabla\zetan(\un)-\nabla\zeta(u)]\geq 0,
\end{equation*}
integrate this expression over $\O\times(0,T)$ and expand. The convergences 
\eqref{eq:zetaweak}, \eqref{eq:anweak}, \eqref{eq:limintan} and the convergence
in $L^{p'}(\O\times(0,T))^{d}$ of $\an(\cdot,\nun(\un),\nabla\zeta(u))$ to $a(\cdot,\nu(u),\nabla\zeta(u))$ imply that, as $n\to\infty$,
\begin{equation*}
\int_{0}^{T}\int_{\O}F_{n}(x,t)\ud x \ud t \to 0.
\end{equation*}
The nonnegativity of $F_n$ then ensures that $F_n$ converges to zero in $L^{1}(\O\times(0,T))$
and therefore, upon extraction of a subsequence, almost everywhere on $\O\times(0,T)$.
Now use the strict monotonicity of $a$ to apply 
Lemma \ref{lem:monotone} with $X=\O\times\RR$, $b_{n}(s,\xi) = \an(x,s,\xi)$, 
$\chi_n = \nabla\zetan(\un)$ to deduce that, up to a subsequence, 
$\nabla\zetan(\un)\to\nabla\zeta(u)$ almost everywhere on $\O\times(0,T)$.
A subsequence of $(\nun(\un))_{n\in\NN}$ converges almost everywhere on $\O\times(0,T)$
to $\nu(u)$, therefore the local uniform convergence on $\RR\times\RR^d$ of $(\an)_{n\in\NN}$ 
ensures that
\begin{equation*}
\an(\cdot,\nun(\un),\nabla\zetan(\un))\cdot\nabla\zetan(\un)\to a(\cdot,\nu(u),\nabla\zeta(u))\cdot\nabla\zeta(u) 
\quad \mbox{a.e. on }\O\times(0,T).
\end{equation*}
Lemma \ref{lem:l1strong} then guarantees, using \eqref{eq:limintan} and the nonnegativity
of $\an(\cdot,\nun(\un),\nabla\zetan(\un))\cdot\nabla\zetan(\un)$, that
\begin{equation*}
\an(\cdot,\nun(\un),\nabla\zetan(\un))\cdot\nabla\zetan(\un)\to a(\cdot,\nu(u),\nabla\zeta(u))\cdot\nabla\zeta(u) 
\quad \mbox{in $L^{1}(\O\times(0,T))$.}
\end{equation*}
Therefore, the sequence $(\an(\cdot,\nun(\un),\nabla\zetan(\un))\cdot\nabla\zetan(\un))_{n\in\NN}$
is equi-integrable, and so too is $(|\nabla\zetan(\un)|^{p})_{n\in\NN}$ thanks to
the uniform coercivity of $(\an)_{n\in\NN}$. The strong convergence \eqref{eq:strgrad}
then follows from Vitali's theorem. \qed


\appendix
\section{Convergence lemmas} \label{app:convlemmas}

We make frequent use of the following lemma, proved in \cite{dt14}.
\begin{lemma} \label{lem:h}
Let $\Hn:\RR\to\RR$ be a sequence of continuous functions such that
\begin{enumerate}[(i)]
\item there exist positive constants $\ctel{h}$, $\gamma$ such that 
for every $s\in\RR$, $|\Hn(s)|\leq \cter{h}(1 + |s|^{\gamma})$;
\item $\Hn$ converges locally uniformly on $\RR$ to a continuous function 
$H:\RR\to\RR$.
\end{enumerate}
Let $N\in\NN$ and take a bounded subset $E$ of $\RR^N$. If $q\in[\gamma,\infty)$ 
and $(\vn)_{n\in\NN}\subset L^{q}(E)$ is such that $\vn\to v$ in $L^{q}(E)$, then
$\Hn(\vn)\to H(v)$ in $L^{q/\gamma}(E)$ as $n\to\infty$.
\end{lemma}
The next lemma gives an equivalent characterisation of uniform convergence, which is
critical to Step 3 of the proof of our main result. For a proof of this lemma, see \cite{dey14}.

\begin{lemma} \label{lem:equiv-unifconv}
Let $(K,d_K)$ be a compact metric space, $(E,d_E)$ a metric space. 
Denote by $\cF(K,E)$ the space of functions $K\to E$, endowed with
the uniform metric $d_\cF(v,w)=\sup_{s\in K}d_E(v(s),w(s))$ (note that
this metric may take infinite values).

Let $(v_n)_{n\in\NN}$ be a sequence in $\cF(K,E)$ and $v:K\to E$ be
continuous. Then $v_n\to v$ for $d_{\cF}$ if and only if, for any $s\in K$ and
any sequence $(s_n)_{n\in\NN}\subset K$ converging to $s$ for $d_K$,
$v_n(s_n)\to v(s)$ for $d_E$.
\end{lemma}
We employ the final two lemmas of this appendix in Section \ref{ssec:step5} to establish the (strong) convergence
in $\LSobo{p}{1}{p} $ of $\zetan(\un)$ to $\zeta(u)$. For a proof of the first of
these lemmas, see \cite[Lemma 3.3]{degh13}. The second is a slight modification of
\cite[Lemma 3.2]{degh13}.

\begin{lemma} \label{lem:l1strong}
Let $(F_{n})_{n\in\NN}$ be a sequence of nonnegative functions in $L^{1}(\O)$. Let
$F\in L^{1}(\O)$ be such that $F_n \to F$ almost everywhere and
\begin{equation*}
\int_{\O}F_{n}(x)\ud x \to \int_{\O}F(x)\ud x.
\end{equation*}
Then $F_{n}\to F$ in $L^{1}(\O)$ as $n\to\infty$.
\end{lemma}

\begin{lemma}\label{lem:monotone}
Let $X$ be a metric space and for every $n\in\NN$ let 
$b_n : X\times\RR^d \to \RR^d$ be continuous and monotone:
\begin{equation*}
(b_{n}(u,\delta)-b_{n}(u,\gamma))\cdot(\delta-\gamma)\geq0 \quad \forall u\in X,\ 
\forall \delta,\gamma\in\RR^d.
\end{equation*}
Assume that $b_n$ converges locally uniformly on $X\times\RR^d$ to a continuous
map $b:X\times\RR^d \to \RR^d$ that is strictly monotone:
\begin{equation*}
(b(u,\delta)-b(u,\gamma))\cdot(\delta-\gamma)>0 \quad \forall u\in X,\ 
\forall \delta\neq\gamma\in\RR^d.
\end{equation*}
Take a sequence $(\un,\chi_{n})\in X\times\RR^d$ and $(u,\chi)\in X\times \RR^d$ such
that as $n\to\infty$,
\begin{equation*}
(b_{n}(\un,\chi_{n})-b_{n}(\un,\chi))\cdot(\chi_{n}-\chi) \to 0 \quad \mbox{and} \quad \un\to u.
\end{equation*}
Then $\chi_n \to \chi$.
\end{lemma}
\begin{proof}
Let $\delta\in\RR^{d}\setminus\{ 0 \}$. For $n\in\NN$, define $h_{\delta,n}:\RR\to\RR$
by
\begin{equation*}
h_{\delta,n}(s) := (b_{n}(\un,\chi+s\delta)-b_{n}(\un,\chi))\cdot\delta.
\end{equation*}
For $s>s'$,
\begin{equation*}
(h_{\delta,n}(s)-h_{\delta,n}(s'))(s-s') = (b_{n}(\un,\chi+s\delta)-b_{n}(\un,\chi+s'\delta))\cdot\delta(s-s')\geq 0,
\end{equation*}
so $h_{\delta,n}$ is a nondecreasing function.
Now assume that $\chi_{n}$ does not converge to $\chi$, so there is some $\eps>0$ and a 
subsequence of $(\chi_{n})_{n\in\NN}$, not relabelled for convenience, such that
$s_{n} := |\chi_{n} - \chi| \geq \eps$ for all $n\in\NN$. 
Define 
\begin{equation*}
\delta_{n} := \frac{\chi_{n} - \chi}{|\chi_{n} - \chi|}.
\end{equation*}
There exists $\delta\in\RR^d$ with $|\delta| = 1$ such that, upon extraction of a 
subsequence, $\delta_n \to \delta$. Then
\begin{equation*}
(b_{n}(\un,\chi_{n}) - b_{n}(\un,\chi))\cdot\frac{\chi_{n}-\chi}{s_{n}} = h_{\delta_{n},n}(s_{n})
 \geq h_{\delta_{n},n}(\eps)
= (b_{n}(\un,\chi+\eps \delta_{n}) - b_{n}(\un,\chi))\cdot\delta_{n}.
\end{equation*}
Let $n\to\infty$ to see that
\begin{align*}
0 &= \lim_{n\to\infty} \frac{1}{s_{n}}(b_{n}(\un,\chi_{n}) - b_{n}(\un,\chi))\cdot(\chi_{n}-\chi)\\
&\geq \lim_{n\to\infty}(b_{n}(\un,\chi+\eps \delta_{n}) - b_{n}(\un,\chi))\cdot\delta_{n}\\
&= (b(u,\chi+\eps\delta) - b(u,\chi))\cdot\delta > 0,
\end{align*}
a contradiction.
\end{proof}


\section{Compensated compactness lemma}\label{app:cc}

Space--time compensated compactness results usually state the convergence
of a product $(f_ng_n)_{n\in\NN}$ of functions, each one converging only weakly but
$(f_n)_{n\in\NN}$ having compactness properties in space and $(g_n)_{n\in\NN}$
having compactness properties in time. As seen in the work of A.V. Kazhikhov \cite{kaz98}, 
A. Moussa \cite{moussa} and references therein, the proof of compensated compactness is often a consequence of
the Aubin--Simon compactness theorem. The following lemma is no exception.

\begin{lemma}\label{lem:compcomp}
Let $\O$ be an open and bounded domain in $\RR^d$, $T>0$, and $p\in (1,\infty)$.
Take two sequences of functions $(f_n)_{n\in\NN}$ and $(g_n)_{n\in\NN}$
in $L^2(\O\times (0,T))$ such that
\begin{itemize}
\item $f_n\weakto f$ and $g_n\weakto g$ weakly-$*$ in $L^2(\O\times(0,T))$ as $n\to\infty$,
\item $(f_n)_{n\in\NN}$ is bounded in $L^p(0,T;W^{1,p}_0(\O))$,
\item $(\partial_t g_n)_{n\in\NN}$ is bounded in $L^{p'}(0,T;W^{-1,p'}(\O))$.
\end{itemize}
Assume furthermore that one of the following properties holds:
\begin{enumerate}[(i)]
\item\label{case1} $p\ge 2$, or
\item\label{case2} $\frac{2d}{d+2}<p<2$ and $(g_n)_{n\in\NN}$ is bounded in
$L^{p'}(0,T;L^2(\O))$, or
\item\label{case3} $1<p\le \frac{2d}{d+2}$, $(g_n)_{n\in\NN}$ is bounded in
$L^{p'}(0,T;L^2(\O))$, and $(f_n)_{n\in\NN}$ is bounded in $L^\infty(\O\times(0,T))$.
\end{enumerate}
Then $f_ng_n\to fg$ in the sense of measures on $\O\times(0,T)$, that is, for
all $\varphi\in C(\overline{\O}\times[0,T])$,
\begin{equation}\label{lem:compcomp.conv}
\int_0^T\int_\O f_n(x,t)g_n(x,t)\varphi(x,t)\ud x\ud t\to 
\int_0^T\int_\O f(x,t)g(x,t)\varphi(x,t)\ud x\ud t\mbox{ as $n\to\infty$}. 
\end{equation}
\end{lemma}

\begin{remark} This result is clearly not optimal and the conclusion holds under
much weaker assumptions. Using for example the ideas of
\cite{dey14}, which consists of reducing the proof to the case where
$(f_n)_{n\in\NN}$ are tensorial functions, we could establish a convergence result for
$(f_ng_n)_{n\in\NN}$ under
weaker bounds on the functions, and assuming only space-translate estimates
of $(f_n)_{n\in\NN}$ instead of bounds in a Lebesgue--Sobolev space. We establish 
only this simpler lemma that is adapted precisely to our setting, and emphasise 
that we make no claim over the originality of its core idea.
\end{remark}

\begin{remark}[$p$ small] If $p$ is too small, then an additional assumption
on $(f_n)_{n\in\NN}$ is mandatory, as the following example shows.

If $p\le \frac{2d}{d+2}$ then
$p^*\le 2$ and $W^{1,p}_0(\O)$ is therefore not compactly embedded in $L^2(\O)$.
Take a sequence $(u_n)_{n\in\NN}$ that is bounded in $W^{1,p}_0(\O)\cap L^2(\O)$
that converges weakly but not strongly to some $u\in L^2(\O)$. Set $f_n(x,t)=g_n(x,t)=u_n(x)$
and $f(x,t)=g(x,t)=u(x)$. Then $f_n\weakto f$ and $g_n\weakto g$ weakly in $L^\infty(0,T;L^2(\O))$, $(f_n)_{n\in\NN}$ is bounded in $L^\infty(0,T;W^{1,p}_0(\O))$
and $\partial_t g_n=0$, but the convergence of $\int_{\O\times (0,T)}f_ng_n$ to $\int_{\O
\times (0,T)}fg$ would imply that $\norm{u_n}_{L^2(\O)}\to \norm{u}_{L^2(\O)}$.
Hence, $u_n$ would converge strongly to $u$ in $L^2(\O)$, which is a contradiction.

\end{remark}

\begin{proof}

By density of $C^\infty(\overline{\O}\times [0,T])$ in $C(\overline{\O}\times [0,T])$,
we only need to prove the result for $\varphi$ smooth. Replacing $(f_n)_{n\in\NN}$
with $(\varphi f_n)_{n\in\NN}$, which has the same bound and convergence properties 
as $(f_n)_{n\in\NN}$, we can actually assume that $\varphi=1$ and we only have to prove
\begin{equation}\label{lem:compcomp.conv2}
\int_0^T\int_\O f_n(x,t)g_n(x,t)\ud x\ud t\to 
\int_0^T\int_\O f(x,t)g(x,t)\ud x\ud t\mbox{ as $n\to\infty$}. 
\end{equation}

We recall a classical consequence of Aubin--Simon's theorem \cite{BF,poly}:
\emph{assume that $V$, $E$ and $F$ are Banach spaces such that $V$ is compactly embedded in $E$ and
$E$ is continuously embedded in $F$; if $(w_n)_{n\in\NN}$ is bounded in $L^r(0,T;V)$
and $(\partial_t w_n)_{n\in\NN}$ is bounded in $L^m(0,T;F)$ for some $r,m\in (1,\infty]$,
then $(w_n)_{n\in\NN}$ is relatively compact in $L^r(0,T;E)$.}

We first consider Cases (\ref{case1}) and (\ref{case2}).
In both cases, $p^*>2$ and thus $W^{1,p}_0(\O)$ is compactly embedded in $L^2(\O)$.
By duality, we infer that $V=L^2(\O)$ is compactly embedded in $E=F=W^{-1,p'}(\O)$.
Since $(g_n)_{n\in\NN}$ is bounded in $L^{p'}(0,T;V)$ (in Case (\ref{case1}),
we use the fact that $p'\le 2$), and $(\partial_t g_n)_{n\in\NN}$ is
bounded in $L^{p'}(0,T;W^{-1,p'}(\O))$, the Aubin--Simon theorem shows that
$(g_n)_{n\in\NN}$ is relatively compact in $L^{p'}(0,T;W^{-1,p'}(\O))$, and that
its convergence to $g$ is strong in this space. 
Since $(f_n)_{n\in\NN}$ is bounded in $L^{p}(0,T;W^{1,p}_0(\O))$,
its convergence to $f$ also holds weakly in this space.
Observe that
\begin{equation*}
\int_0^T\int_\O f_n(x,t)g_n(x,t)\ud x\ud t=
\int_0^T \langle g_n(t),f_n(t)\rangle_{W^{-1,p'},W^{1,p}_0}\ud t,
\end{equation*}
so the convergence \eqref{lem:compcomp.conv2} holds by strong/weak convergence.

We now consider Case (\ref{case3}). Fix $s\in (0,1)$
such that $2s<p$. By the assumptions on $(f_n)_{n\in\NN}$ and
Lemma \ref{lem:interpWL} below, the sequence $(f_n)_{n\in\NN}$ is
bounded in $L^p(0,T;W^{s,2}_0(\O))$,
and thus converges weakly in this space to $f$.
Since $s>0$, $W^{s,2}_0(\O)$ is compactly embedded in $L^2(\O)$
(we use \cite[Theorem 7.1]{NPV11} with the extension $W^{s,2}_0(\O)\to W^{s,2}(\RR^d)$
by $0$ outside $\O$, which is valid since $W^{s,2}_0(\O)$ is the closure in $W^{s,2}(\O)$
of compactly supported functions). Dually, $V=L^2(\O)$ is compactly embedded in 
$E=W^{-s,2}(\O)$. Set $F=W^{-s,2}(\O)+W^{-1,p'}(\O)$,
and apply the Aubin--Simon theorem to see that $(g_n)_{n\in\NN}$
is relatively compact in $L^{p'}(0,T;W^{-s,2}(\O))$. The weak convergence 
of $(f_n)_{n\in\NN}$ in $L^p(0,T;W^{s,2}_0(\O))$ therefore allows us to pass to the 
weak/strong limit in \eqref{lem:compcomp.conv2} as above. \end{proof}

The following lemma is a simple interpolation result between $W^{1,p}_0(\O)$ and 
$L^\infty(\O)$.

\begin{lemma}[Interpolation estimate]\label{lem:interpWL}
Let $\O$ be a bounded open subset of $\RR^d$
and $p\in (1,\infty)$. If $s\in (0,1)$ and $q\in (p,\infty)$ are such that $sq<p$,
then there exists $\ctel{cst:interpWL}$ such that for all $w\in W^{1,p}_0(\O)$
\begin{equation*}
\forall w\in W^{1,p}_0(\O)\cap L^\infty(\O)\,,\;
||w||_{W^{s,q}_0(\O)}\le \cter{cst:interpWL}||w||_{L^\infty(\O)}^{1-\frac{p}{q}}
||w||_{W^{1,p}_0(\O)}^{\frac{p}{q}},
\end{equation*}
where $W^{s,q}_0(\O)$ is the closure in $W^{s,q}(\O)$ (for the norm defined in
the proof) of $C^\infty_c(\O)$.
\end{lemma}

\begin{proof} 
We write, using the change of variable $y=x+\xi$,
\begin{align*}
||w||_{W^{s,q}_0(\O)}^q&=\int_\O |w(x)|^q\ud x + \int_\O\int_\O\frac{|w(x)-w(y)|^q}{|x-y|^{d+sq}}\ud x\ud y\\
&\le ||w||_{L^\infty(\O)}^{q-p}||w||_{L^p(\O)}^p +2||w||_{L^\infty(\O)}^{q-p}
\int_\O\int_\O\frac{|w(x)-w(y)|^p}{|x-y|^{d+sq}}\ud x\ud y\\
&\le ||w||_{L^\infty(\O)}^{q-p}||w||_{L^p(\O)}^p +2||w||_{L^\infty(\O)}^{q-p}
\int_{\O-\O}|\xi|^{-d-sq}\left(\int_\O |w(x+\xi)-w(x)|^p\ud x\right)\ud\xi.
\end{align*}
But $\int_\O |w(x+\xi)-w(x)|^p\ud x\le |\xi|^p||\nabla w||_{L^p(\O)^d}^p$ and
$\O-\O\subset B(0,D)$ where $D$ is the diameter of $\O$. Hence
\[
||w||_{W^{s,q}_0(\O)}^q
\le ||w||_{L^\infty(\O)}^{q-p}||w||_{L^p(\O)}^p +2||w||_{L^\infty(\O)}^{q-p}
||\nabla w||_{L^p(\O)^d}^p\int_{B(0,D)}|\xi|^{p-d-sq}\ud\xi\\
\]
and the proof is complete since $p-d-sq>-d$. \end{proof}


\section{A uniqueness result} \label{app:unicite}

We state and prove the uniqueness of a solution to \eqref{eq:model} when $p=2$ and
\begin{equation}
\begin{aligned} 
& a(x,\nu(u),\nabla\zeta(u)) = \Lambda(x) \nabla\zeta(u) &\mbox{in } \Omega,
\end{aligned} \label{eq:modelsimp}
\end{equation}
under the hypothesis that
\begin{equation}
\begin{array}{llllll}\displaystyle
\Lambda \hbox{ is a measurable function from } \Omega  \hbox{ to } {{ \mathcal M}_d(\RR)}\mbox{ and }\\
\mbox{ there exists $\underline{\lambda},\overline{\lambda}>0$ such that, for a.e. $ x\in\Omega$,}\\
\mbox{ $\Lambda( x)$ is symmetric with eigenvalues in $[\underline{\lambda},\overline{\lambda}]$.}
\end{array}
\label{hyplambda}
\end{equation}
J. Carillo \cite{carr1999} gave a proof, based on the doubling variable technique, of the uniqueness of entropy solutions to 
$\partial_t \beta(u) - \Delta\zeta(u) = f$ (with an additional convective term). Although this could be extended to our framework,
we provide here another proof which is shorter and simpler, using the idea due to J. Hadamard \cite{had23} of solving the dual problem. 
This idea has been succesfully used in the case of the one-dimensional Stefan problem \cite{bkp85}, 
and subsequently generalized to the higher dimensional case \cite{ghp97}.
 
Note that this uniqueness result applies to the equivalent maximal monotone graph formulation \eqref{eq:modelmonot}--\eqref{eq:modelsimp}--\eqref{hyplambda}.

\begin{theorem}\label{parnlthunicite}
Under Hypotheses  \eqref{assumptions}, \eqref{eq:modelsimp} and \eqref{hyplambda}, let $u_1$ and $u_2$ be two solutions to \eqref{eq:solution} in the sense of Definition \ref{def:solution}. Then $\beta(u_1) = \beta(u_2)$ and $\zeta(u_1) = \zeta(u_2)$.
\end{theorem}

\begin{remark}If $\beta$ and $\zeta$ do not have any common plateau, as a corollary
of this theorem we see that $u_1=u_2$. Otherwise,
Theorem \ref{parnlthunicite} is optimal. Indeed, whenever a solution $u$ takes a value in a common
plateau of $\beta$ and $\zeta$,
we can change this value into any other value in the same plateau without changing the fact
that $u$ is a solution.
\end{remark}

\begin{proof}
Set $u_d = \beta(u_1) + \zeta(u_1) - \beta(u_2)- \zeta(u_2)$, and 
for all $( x,t)\in \O\times [0,T]$, define 
\[
q( x,t) = \left\{\begin{array}{ll}\frac{\zeta(u_1(x,t)) - \zeta(u_2(x,t))}{u_d(x,t)}&\mbox{ if
$u_d( x,t) \neq 0$},\\
0&\mbox{ otherwise}.\end{array}\right.
\]
Take
$\psi \in L^2(0,T;H^1_0(\O))$ with  $\partial_t \psi\in L^2(\O\times(0,T))$, $\psi(\cdot,T)=0$
and $\dive(\Lambda\nabla \psi)\in L^2(\O\times(0,T))$. Approximating
$\psi$ in time by smooth functions, we see that 
\[
\int_0^T \langle \partial_t\beta(u_i),\psi\rangle
=-\int_\O \beta(\uini)(x)\psi(x,0)\ud x-\int_0^T \beta(u_i)(x,t)\partial_t \psi(x,t)\ud x\ud t.
\] 
Then \eqref{eq:solution} gives
\begin{equation}
\int_0^T\int_\Omega u_d( x,t) \Bigl((1-q(x,t)) \partial_t \psi( x,t) + q( x,t) \dive(\Lambda\nabla \psi)( x,t)\Bigr) \ud x \ud t  = 0.
\label{parnlun1}\end{equation}
For $\eps\in (0,1/2)$, denote $q_\eps = (1-2\eps)q+ \eps$. Since $0 \le q\le 1$ we have
$\eps \leq q_\eps \leq 1-\eps$
and
\begin{equation}
 \frac{(q_\eps-q)^2} {q_\eps} \le  \eps\quad \hbox{ and }\quad\frac{(q_\eps-q)^2} {1 -q_\eps}\le  \eps.
\label{parnlun3}\end{equation}
Let $\psi_\eps$ be given by Lemma \ref{parnlauxi} below, with $g=q_\eps$
and some $w\in C^\infty_ c(\O\times (0,T))$.
Substituting $\psi$ by $\psi_\eps$ in (\ref{parnlun1}) and using  \eqref{parnllemme1},
\begin{multline}
\left| \int_0^T\int_\Omega u_d( x,t) w( x,t)\ud x \ud t\right|\\ 
\le \left|\int_0^T\int_\Omega u_d( x,t)(q_\eps( x,t)-q( x,t)) (\dive(\Lambda\nabla \psi_\eps)( x,t) - \partial_t \psi_\eps( x,t))\ud x \ud t \right|.
\label{parnlun10}\end{multline}
The Cauchy-Schwarz inequality, \eqref{parnllemme5} and  \eqref{parnlun3} imply that
\begin{multline}
\Bigg[\int_0^T \int_\Omega u_d(x,t)(q_\eps(x,t) - q(x,t))(\dive(\Lambda\nabla \psi_\eps)(x,t) - \partial_t \psi_\eps(x,t))\ud x\ud t\Bigg]^2 \\
\leq 2\left(\int_0^T \int_\Omega u_d( x,t)^2 \frac{(q(x,t)-q_\eps(x,t))^2}{q_\eps(x,t)} \ud x\ud t\right)
\left(\int_0^T \int_\Omega q_\eps(x,t)\Bigl(\dive(\Lambda\nabla\psi_\eps)(x,t)\Bigr)^2 \ud x\ud t \right)\\ 
+ 2\left(\int_0^T \int_\Omega u_d(x,t)^2\frac{(q(x,t)-q_\eps(x,t))^2}{1 - q_\eps(x,t)} \ud x\ud t \right)
\left(\int_0^T \int_\Omega (1-q_\eps(x,t))\Bigl(\partial_t\psi_\eps(x,t)\Bigr)^2 \ud x \ud t\right) \\
\leq \eps C_0\left(\norm{\nabla w}_{L^2(\O\times(0,T))}^2 + \norm{w}_{L^2(\O\times(0,T))}^2  
+ \norm{\partial_t w}_{L^2(\O\times(0,T))}^2\right)\int_0^T \int_\Omega u_d(x,t)^2 \ud x \ud t.
\label{parnlun11}
\end{multline}
The right-hand side
of (\ref{parnlun11}) vanishes as $\eps\to 0$, and therefore so does
left-hand side of (\ref{parnlun10}), giving
\[
 \int_0^T\int_\Omega u_d( x,t) w( x,t) \ud x \ud t  =0.
\]
Since this holds for any function
$w\in C^\infty_ c(\O\times (0,T))$, we get that $u_d( x,t) = 0$ for a.e.\ $( x,t)\in\Omega\times (0,T)$. Hence $\beta(u_1)-\beta(u_2)=-(\zeta(u_1)-\zeta(u_2))$ and,
since $\beta$ and $\zeta$ are non-decreasing, the proof of the theorem is complete.
\end{proof}

The following lemma ensures the existence of the function $\psi$, used in the proof of Theorem \ref{parnlthunicite}. 

\begin{lemma}\label{parnlauxi}
Let $T>0$, and let $\O$ be a bounded open subset of $\RR^d$ ($d\in\NN$). Assume Hypothesis \eqref{hyplambda}. Let $w\in C^\infty_c(\Omega\times (0,T))$ and $g\in L^\infty(\O\times(0,T))$ such that 
$g( x,t)\in [{g_{\rm min}},1 -{g_{\rm min}}]$ for a.e.\ $( x,t)\in \O\times(0,T)$,
where ${g_{\rm min}}$ is a fixed number in $(0,\frac 1 2)$. Then there exists a function $\psi$ such that:
\begin{enumerate}[(i)]
\item $\psi\in L^\infty(0,T;H^1_0(\O))$, $\partial_t \psi\in L^2(\O\times(0,T))$, $\dive(\Lambda\nabla \psi)\in L^2(\O\times(0,T))$ (this implies $\psi\in C([0,T];L^2(\O))$),
\item $\psi(\cdot,T) = 0$,
\item $\psi$ satisfies
\begin{equation}
(1 - g(x,t)) \partial_t \psi( x,t) + g( x,t) \dive(\Lambda\nabla \psi)( x,t) = w( x,t)
\quad\mbox{for a.e.\ $(x,t)\in\O\times(0,T)$,}\label{parnllemme1}
\end{equation}
\item there exists $C_0>0$, depending only on $T$, $\diam{\O}$, $\underline{\lambda}$ and $\overline{\lambda}$ (and \emph{not} on $g_{\rm min}$), such that
\begin{multline}
\int_0^T \int_\Omega \left((1 - g( x,t)) \Bigl(\partial_t \psi( x,t)\Bigr)^2 + g( x,t) \Bigl(\dive(\Lambda\nabla\psi)( x,t)\Bigr)^2 \right)\ud x \ud t\\
\le  C_0\left(\norm{\nabla w}_{L^2(\O\times(0,T))}^2 + \norm{w}_{L^2(\O\times(0,T))}^2  + \norm{\partial_t  w}_{L^2(\O\times(0,T))}^2\right).
\label{parnllemme5}\end{multline}
\end{enumerate}
\end{lemma}

\begin{proof}
After dividing through by $g$, observe that \eqref{parnllemme1} is equivalent to
\begin{equation}\label{eq:equivpde}
\Phi(x,t)\partial_t \psi(x,t) + \dive(\Lambda(x)\nabla \psi(x,t)) = f(x,t),
\end{equation}
where $\Phi\in L^{\infty}(\O\times(0,T))$ satisfies $0 < \phi_\ast \leq \Phi(x,t) \leq \phi^\ast$ for
a.e. $(x,t)\in\O\times(0,T)$ and $f\in L^{\infty}(\O\times(0,T))$. We first show the existence
of a solution $\psi$ to \eqref{eq:equivpde} satisfying (i) and (ii).

Let $W := \{ v\in C^0([0,T]; H^1_0(\O)) \,|\, \partial_t v \in L^2(\O\times(0,T))\mbox{ and
$v(\cdot,T)=0$} \}$. Define
$T: W \to W$, where $T(v) = u$ is such that for all $w\in L^2(0,T; H^1_0(\O))$,
\begin{multline}\label{eq:constantcoeff}
\int_0^T\int_\O\left( \phi^\ast w(x,t)\partial_t u(x,t) - \Lambda(x)\nabla u(x,t)\cdot\nabla w(x,t)\right)\ud x \ud t\\
= \int_0^T\int_\O\left(f(x,t) + (\phi^\ast - \Phi(x,t))\partial_t v(x,t)\right)w(x,t)\ud x\ud t.
\end{multline}
Existence of such a $u\in W$ is assured thanks to Lemma \ref{parnlauxi2} below.
Endowing $W$ with the norm
\[
\norm{v}_{W} := \left[\sup_{\tau\in[0,T]}\left(\norm{\partial_t v}_{L^2(\O\times(\tau,T))}^2+\frac{\underline{\lambda}}{\phi^\ast }\norm{\nabla v(\cdot,\tau)}_{L^2(\O)^d}^2\right)\right]^{1/2},
\]
dividing \eqref{eq:constantcoeff} by $\phi^\ast$, noticing that $\sup_{(x,t)\in \O\times (0,T)}
\left|\frac{\phi^\ast-\Phi(x,t)}{\phi^\ast}\right|\le \frac{\phi^\ast-\phi_\ast}{\phi^\ast}<1$
and using \eqref{est.dtu}, we see that $T$ is a contraction. It therefore has a unique fixed point $\psi\in W$
that satisfies (i)--(iii).

It remains to verify \eqref{parnllemme5}. Taking $s,\tau\in[0,T]$, we have
\[
\int_s^\tau \int_\Omega w( x,t)\dive(\Lambda\nabla \psi)( x,t)\ud x\ud t = -\int_s^\tau \int_\Omega \Lambda(x)\nabla w( x,t)\cdot \nabla \psi( x,t)\ud x\ud t,
\]
and
\begin{multline*}
\int_s^\tau \int_\Omega w(x,t) \partial_t \psi(x,t)\ud x\ud t = \int_\Omega  (w(x,\tau) \psi(x,\tau) - w(x,s) \psi(x,s))\ud x \\
-\int_s^\tau \int_\Omega \psi(x,t)\partial_t w(x,t) \ud x\ud t.
\end{multline*}
Multiplying \eqref{parnllemme1} by $\partial_t \psi( x,t)+\dive(\Lambda\nabla \psi)( x,t)$, integrating on $\O\times (s,T)$ for $s\in [0,T]$, and using \eqref{est.dtudiv},
$ \psi(\cdot,T) = 0$ and $\nabla \psi(\cdot,T) = 0$, we obtain
\begin{multline}
{\frac 1 2}\int_\Omega  \Lambda(x)\nabla \psi( x,s)\cdot\nabla \psi( x,s) \ud x  \\
+ \int_s^T \int_\Omega \left((1 - g( x,t)) \Bigl(\partial_t \psi( x,t)\Bigr)^2 + g( x,t) \Bigl(\dive(\Lambda\nabla\psi)( x,t)\Bigr)^2 \right) \ud x \ud t  \\ 
= - \int_s^T \int_\Omega \Lambda(x)\nabla w( x,t) \cdot \nabla \psi( x,t) \ud x \ud t-\int_\Omega  w(x,s) \psi(x,s)\ud x -\int_s^T \int_\Omega \psi(x,t)\partial_t w(x,t) \ud x\ud t.
\label{parnllemme7}\end{multline}
Integrating \eqref{parnllemme7} with respect to $s\in (0,T)$ leads to
\begin{multline}
{\frac 1 2}\int_0^T \int_\Omega  \Lambda(x)\nabla \psi( x,s)\cdot\nabla \psi( x,s)  \ud x \ud s \le  
T \int_0^T \int_\Omega \vert\Lambda(x)\nabla w( x,t) \cdot \nabla \psi( x,t)\vert \ud x \ud t \\ + \int_0^T\int_\Omega  \vert w(x,s) \psi(x,s)\vert\ud x\ud s + T\int_0^T \int_\Omega \vert \psi(x,t)\partial_t w(x,t)\vert \ud x\ud t.
\label{parnllemme8}\end{multline}
We then apply the Cauchy-Schwarz and Poincar\'e inequalities, which leads to
\begin{multline}
{\frac {\underline{\lambda}} 2}\norm{\nabla \psi}_{L^2(\O\times(0,T))}\\
\le  
T \overline{\lambda} \norm{\nabla w}_{L^2(\O\times(0,T))} 
+ \diam{\O}\left(\norm{w}_{L^2(\O\times(0,T))}  + T\norm{\partial_t  w}_{L^2(\O\times(0,T))}\right).
\label{parnllemme9}\end{multline}
Letting $s=0$ in \eqref{parnllemme7}, recalling that $w(\cdot,0) = 0$, and using \eqref{parnllemme9} gives
\begin{align*}
\int_0^T \int_\Omega \Bigl((1 - &g( x,t)) (\partial_t \psi( x,t))^2 + g( x,t) (\dive\Lambda\nabla\psi( x,t))^2 \Bigr) \ud x \ud t  \\
\le {}&
\frac 2 {\underline{\lambda}}\left(\overline{\lambda}\norm{\nabla w}_{L^2(\O\times(0,T))} + \diam{\O}\norm{\partial_t  w}_{L^2(\O\times(0,T))}\right) \\ 
&\times\left(T\overline{\lambda} \norm{\nabla w}_{L^2(\O\times(0,T))} + \diam{\O}  (\norm{w}_{L^2(\O\times(0,T))}  + T\norm{\partial_t  w}_{L^2(\O\times(0,T))}\right),
\end{align*}
which implies \eqref{parnllemme5}.
\end{proof}

The following lemma states the time regularity of the solution of a linear backwards
parabolic problem with sufficiently regular data.
It may be that this lemma can be proved by using the Hille--Yoshida theorem,
since the regularity of the solution is coherent with those of the Hille--Yoshida theory,
but this exact result, with low regularity assumptions on $\O$ or $\Lambda$,
does not seem to exist in the literature. We propose a self-contained proof, which is
probably shorter than checking that the Hille--Yoshida framework applies. 

\begin{lemma}\label{parnlauxi2}
Let $T>0$, and let $\O$ be a bounded open subset of $\RR^d$ ($d\in\NN$). Assume Hypothesis \eqref{hyplambda}. Let $h\in L^2(\Omega\times (0,T))$, and let $u\in L^2(0,T;H^1_0(\O))$ with $\partial_t u\in L^2(0,T;H^{-1}(\O))$ such that $u(\cdot,T) = 0$ be the standard weak solution of
the backwards problem $\partial_t u + \dive(\Lambda\nabla u)=h$, that is
\begin{multline}
 \forall v\in L^2(0,T;H^1_0(\O)),\\
\int_0^T\left(\langle \partial_t u(\cdot,t),v(\cdot,t)\rangle_{H^{-1},H^1_0} - \int_\O \Lambda(x) \nabla u(x,t)\cdot\nabla v(x,t)\ud x\right) \ud t =  \int_0^T \int_\O  h(x,t) v(x,t)\ud x \ud t.
\label{eq:parabstd}\end{multline}
Then $\partial_t u \in L^2(\Omega\times (0,T))$, $\dive(\Lambda\nabla u) \in L^2(\Omega\times (0,T))$, $u\in C^0([0,T];H^1_0(\O))$ and
\begin{equation}\label{est.dtu}
\sup_{\tau\in [0,T]}
\left(\norm{\partial_t u}_{L^2(\O\times (\tau,T))}^2+\int_\O \Lambda(x)
\nabla u(x,\tau)\cdot\nabla u(x,\tau)\ud x\right)\le \norm{h}_{L^2(\O\times (0,T))}^2.
\end{equation}
Furthermore, for all $s<\tau\in [0,T]$,
\begin{eqnarray} \label{est.dtudiv}
\int_s^\tau \int_\Omega \partial_t u( x,t)\dive(\Lambda\nabla u)( x,t)\ud x\ud t = -\frac 1 2 \int_\Omega \Lambda(x)\nabla u( x,\tau)\cdot\nabla u( x,\tau)\ud x\\
+ \frac 1 2 \int_\Omega  \Lambda(x)\nabla u( x,s)\cdot\nabla u( x,s)\ud x\nonumber.
\end{eqnarray}
\end{lemma}

\begin{proof} Let $\rho\in C^\infty_c(\RR)$ with support in $[-1,0]$, such that $\rho\ge 0$ and $\int_{-1}^0 \rho(s)\ud s = 1$. For $n\in\NN$, define $\rho_n(s) = n\rho(ns )$
and $u_n(x,t) = \int_0^T  \rho_n(t - s) u(x,s)\ud s$. Take
\[
 v(x,t) = \int_0^T  \rho_n(s - t) \partial_t u_n(x,s)\ud s
\]
as test function in \eqref{eq:parabstd}. Since $v$ is a regular function with respect to time and $v(\cdot,0) = 0$ (thanks to the support of $\rho$), we obtain  $T_1 + T_2 = T_3$, with
\[
 T_1 = \int_0^T\int_\O u(x,t) \int_0^T  \rho_n'(s - t)\partial_t u_n(x,s)\ud s\ud x\ud t = \int_0^T  \int_\O(\partial_t u_n(x,s))^2\ud x \ud s,
\]
\begin{align*}
 T_2 &= -\int_0^T\int_\O \Lambda(x) \nabla u(x,t)\cdot\int_0^T  \rho_n(s - t)\nabla \partial_t u_n(x,s)\ud s\ud x \ud t\\
 &= -\frac 1 2 \int_0^T \int_\O  \partial_t(\Lambda(x)\nabla u_n(x,\cdot)\cdot \nabla u_n(x,\cdot))(s)\ud x\ud s\\ 
&= -\frac 1 2 \left(\int_\O \Lambda(x)\nabla u_n(x,T)\cdot \nabla u_n(x,T)\ud x - \int_\O \Lambda(x)\nabla u_n(x,0)\cdot \nabla u_n(x,0)\ud x\right)
\end{align*}
and
\[
 T_3 = \int_0^T \int_\O  h(x,t) \int_0^T  \rho_n(s - t)\partial_t u_n(x,s)\ud s\ud x \ud t =  \int_0^T  \int_\O h_n(x,s) \partial_t u_n(x,s)\ud s\ud x,
\]
where $h_n(x,s) = \int_0^T  h(x,t)  \rho_n(s - t)\ud t$. Observing that $u_n(\cdot,T) = 0$ and $\Vert h_n\Vert_{L^2(\O\times(0,T))} \le \Vert h\Vert_{L^2(\O\times(0,T))}$, Young's inequality on the right-hand side of $T_1+T_2=T_3$ yields
\begin{equation}\label{est.dtu.1}
 \int_0^T  \int_\O(\partial_t u_n(x,s))^2\ud x \ud s 
+ \int_\O \Lambda(x)\nabla u_n(x,0)\cdot \nabla u_n(x,0)\ud x
\le  \Vert h\Vert_{L^2(\O\times(0,T))}^2.
\end{equation}
Hence, $(\partial_t u_n)_{n\in\NN}$ is bounded in $L^2(\O\times (0,T))$.
Since $u_n\to u$ in $L^2(\O\times (0,T))$, this shows that $\partial_t u\in L^2(\O\times (0,T))$.
The PDE \eqref{eq:parabstd} gives $\dive(\Lambda\nabla u)=\partial_t u-h\in L^2(\O\times (0,T))$.
Finally, we obtain \eqref{est.dtu} by repeating the reasoning leading
to \eqref{est.dtu.1}, starting with
an arbitrary time $\tau\in [0,T]$ instead of $0$ and by passing to the weak limits
in the corresponding inequalities. Equation \eqref{est.dtudiv} is established by repeating the above computations using the same regularization in time.

\end{proof}

\thanks{\textbf{Acknowledgements}: The authors thank Thierry Gallou\"et and Cl\'ement Canc\`es for fruitful discussions on compensated compactness results.}


\bibliography{stabib}{}
\bibliographystyle{plain}

\end{document}